\documentclass[12pt,reqno]{amsart}
\usepackage{amsmath,amsthm,amsfonts,amssymb,amscd,amstext}
\usepackage[ansinew]{inputenc}
\usepackage[dvips]{graphicx}
\usepackage{psfrag,euscript}
\usepackage{a4wide}

\usepackage{hyperref}
\usepackage{color}
\definecolor{darkred}{rgb}{0.4,0,0}
\definecolor{darkgreen}{rgb}{0,0.5,0}
\definecolor{darkblue}{rgb}{0,0,0.4}

\hypersetup{
    pdftitle={Holder-continuity of Oseledets bundles},	% title
    pdfauthor={Araujo, Bufetov, Filip},	% author
    pdfsubject={math},		% subject of the document
    pdfkeywords={},	% list of keywords
    pdfnewwindow=true,		% links in new window
    colorlinks=true,		% false: boxed links; true: colored links
    linkcolor=darkblue,		% color of internal links (change box color with linkbordercolor)
    citecolor=darkred,		% color of links to bibliography
    filecolor=darkblue,		% color of file links
    urlcolor=darkblue,		% color of external links
    pdfborder={0 0 0},
    breaklinks=true
}

\usepackage{titletoc}

\let\oldtocsection=\tocsection
\let\oldtocsubsection=\tocsubsection
\let\oldtocsubsubsection=\tocsubsubsection

\renewcommand{\tocsection}[2]{\hspace{0em}\oldtocsection{#1}{#2}}
\renewcommand{\tocsubsection}[2]{\hspace{1.75em}\oldtocsubsection{#1}{#2}}
\renewcommand{\tocsubsubsection}[2]{\hspace{2em}\oldtocsubsubsection{#1}{#2}}

\newcommand{\bR}{{\mathbb R}}

\newcommand{\cM}{{\mathcal M}}
\newcommand{\cH}{{\mathcal H}}

\newcommand{\norm}[1]{\left\|#1\right\|}
\newcommand{\SL}{\operatorname{SL}}

\newcommand{\onto}{{\twoheadrightarrow}}
\usepackage{pxfonts}   %Nice fonts! The pxfonts package must be loaded AFTER amssymb, amsmath etc
\numberwithin{equation}{section}

\newcommand{\qand}{\quad\text{and}\quad}

\theoremstyle{plain}
\newtheorem{maintheorem}{Theorem}
\newtheorem{maincorollary}[maintheorem]{Corollary}

\newtheorem{theorem}{Theorem}[section]

\newtheorem{lemma}[theorem]{Lemma}

\theoremstyle{definition}
\newtheorem{remark}[theorem]{Remark}

\newcommand{\ang}[1]{\operatorname{angle\,}(#1)}

\newcommand{\RR}{{\mathbb R}}
\newcommand{\CC}{{\mathbb C}}

\newcommand{\NN}{{\mathbb N}}
\newcommand{\ZZ}{{\mathbb Z}}

\newcommand{\vfi}{\varphi}

\newcommand{\diam}{\operatorname{diam}}
\renewcommand{\epsilon}{\varepsilon}
\newcommand{\dist}{\operatorname{dist}}

\newcommand{\cP}{\EuScript{P}}

\newcommand{\cV}{\EuScript{V}}

\title{On H\"older-continuity of Oseledets subspaces}

\begin{thanks} {V.A. was partially supported by CNPq, FAPERJ
    and PRONEX (Brazil). A.I.B.'s research work has been
    carried out thanks to the support of the A*MIDEX project
    (no. ANR-11-IDEX-0001-02) funded by the programme
    ``Investissements d'Avenir '' of the Government of the
    French Republic, managed by the French National Research
    Agency (ANR). A.I.B.'s research was also supported in
    part by the Grant MD-2859.2014.1 of the President of the
    Russian Federation, by the Programme ``Dynamical systems
    and ergodic theory'' of the Presidium of the Russian
    Academy of Sciences, by a subsidy granted to the HSE by
    the Government of the Russian Federation for the
    implementation of the Global Competitiveness Program and
    by the RFBR grant 13-01-12449.  S.F. gratefully
    acknowledges the hospitality of Aix-Marseille
    University.}
\end{thanks}

\author{V\'{\i}tor Ara\'ujo}

\address{ 
\parbox{0.9\textwidth}{
V\'\i tor Ara\'ujo\\
 Instituto de Matem\'a\-tica,
  Universidade Federal da Bahia,\\
Av. Ademar de Barros s/n, 40170-110 Salvador, Brazil.}}
\email{vitor.araujo.im.ufba@gmail.com,
  vitor.d.araujo@ufba.br}
\urladdr{www.sd.mat.ufba.br/$\sim$vitor.d.araujo}

\author{Alexander I. Bufetov}
\address{
\parbox{0.9\textwidth}{
Alexander I. Bufetov\\
Aix-Marseille University, Centrale Marseille, CNRS, I2M, UMR
7373, 39\\
Rue Joliot Curie 13453, Marseille\\
France
\\
and
\\
The Steklov Institute of Mathematics, 119991 Moscow,
Russia\\
and\\
The Institute for Information Transmission Problems, 127051
Moscow. Russia\\
and\\
National Research University Higher School of Economics,
101000 Moscow. Russia}
}
\email{bufetov@mi.ras.ru,
alexander.bufetov@univ-amu.fr}

\author{Simion Filip}
\address{
\parbox{0.9\textwidth}{
Simion Filip\\
Department of Mathematics,
University of Chicago\\
5734 S. University Avenue, Chicago IL, 60637, USA}
}
\email{sfilip@math.uchicago.edu}

\date{\today}

\begin{document}


\renewcommand{\subjclassname}{\textup{2000} Mathematics Subject Classification}

%\keywords{}

\begin{abstract} For H\"older cocycles over a Lipschitz base
  transformation, possibly non-invertible, we show that the
  subbundles given by the Oseledets Theorem are
  H\"older-continuous on compact sets of measure arbitrarily
  close to $1$.  The results extend to vector bundle
  automorphisms, as well as to the Kontsevich-Zorich cocycle
  over the Teichm{\"u}ller flow on the moduli space of
  abelian differentials.  Following a recent result of
  Chaika-Eskin, our results also extend to any given
  Teichm\"uller disk.
\end{abstract}

\maketitle

\tableofcontents

\section{Introduction}
\label{sec:introd}

The Oseledets Multiplicative Ergodic Theorem
\cite{Os68} provides the theoretical background for
computation of Lyapunov exponents of a nonlinear dynamical
system. These exponents define subspaces of vectors having
the same exponential rate of growth under the action of a
\emph{typical} family of linear maps generated by the orbits
of a measure preserving system; typical here meaning almost
every point, see below for the precise statements for discrete
dynamical systems and flows.

The dependence of the exponents and the corresponding
subspaces on the orbit is measurable in general; strong
forms of dependence being rarely studied. For the special
case of an ergodic partially hyperbolic probability measure
preserved by a $C^{1+\alpha}$ diffeomorphism of a compact
manifold, it is known that the subspace given by the direct
sum of the Oseledets subspaces corresponding to strictly
negative Lyapunov exponents depend H\"older-continuously on
the orbit chosen; see Brin~\cite{brin-app} and \cite[Chpt.5,
Section 3]{BarPes2007}.

Here we show that all Oseledets subspaces depend
H\"older-continuously on the points of hyperbolic
blocks (or \emph{regular sets}) of the system, which can be
chosen to be compact sets of measure arbitrarily close to
$1$. We present our results for cocycles generated by
discrete maps (both invertible and non-invertible) and also
for vector bundle automorphisms covering flows, under some
regularity conditions on the underlying dynamics.

In addition, we show that our results apply to the
Kontsevich-Zorich cocycle over the Teichm{\"u}ller flow on
the moduli space of abelian differentials.  Following a
recent result of Chaika-Eskin \cite{Chaika-Eskin}, our
results also extend to any given Teichm\"uller disk.

We begin, in the following subsections, by stating precisely
our results in the setting of the Multiplicative Ergodic
Theorem for cocycles over measure preserving Lipschitz maps;
then for cocycles over measure preserving Lipschitz flows
(Subsections 1.1 and 1.2); and finally stating our results
in the setting of Kontsevich-Zorich cocycle over the
Teichm{\"u}ller flow (Subsection 1.3).

For the proof, we present some preliminary results in
Section \ref{sec:prelim-results-defin} together with the
statements of the main technical lemmas. Then in
Section~\ref{sec:holder-contin-oseled} we prove the
sufficient conditions for the Oseledets subspaces to depend
H\"older continuously on points of hyperbolic
blocks. Section~\ref{sec:lemmata} is reserved to the
detailed proofs of the technical lemmas and we finally apply
our results to the Kontsevich-Zorich cocycle in
Section~\ref{sec:KZ}.

\subsection{Multiplicative ergodic theorem for cocycles}

Let $f:M\to M$ be a measurable transformation on the metric
space $(M,d)$, preserving some Borel probability measure
$\mu$, and let $A:M\to \operatorname{GL}(d,\RR)$ (or
$\operatorname{GL}(d,\CC)$) be any measurable function such
that $\log^+\|A(x)\|$ is $\mu$-integrable, where $\|\cdot\|$
is any norm in the space of real (or complex) $d\times d$
matrices and $\log^+(a):=\max\{\log a,0\}$ for $a>0$.  The
Oseledets theorem states that Lyapunov exponents exist for
the sequence $A^n(x)=A(f^{n-1}(x)) \cdots A(f(x)) \, A(x)$
for $\mu$-almost every $x\in M$. More precisely, for
$\mu$-almost every $x\in M$ there exists $k=k(x)\in\{1,
\ldots, d\}$, a filtration
$$
\{0\}=F^0_x\subset F^1_x \subset \cdots \subset F_x^{k-1}
\subset F_x^k=\RR^d (\text{or $\CC^d$}),
$$
and numbers $\chi_1(x)<\cdots<\chi_k(x)$ such that
\begin{align*}
  F^i_x&=\left\lbrace v\in\RR^d: \lim_{n\to\infty} \frac 1n
  \log\|A^n(x)\cdot v\| \le \chi_i(x)\right\rbrace, \qand
  \\
  \lim_{n\to\infty} &\frac 1n \log\|A^n(x)\cdot v\| =
  \chi_i(x) \quad\text{for all $v\in F_x^i\setminus
    F_x^{i-1}$ and $i\in\{1,\ldots, k\}$}
\end{align*}
where $\|\cdot\|$ is a norm in $\RR^d$ ($\CC^d$).  More
generally, this conclusion holds for any vector bundle
automorphism $A:\cV\to\cV$ over the transformation $f$, with
$A(x):\cV_x\to\cV_{f(x)}$ denoting the action of the
bijective linear map between the fibers $\cV_x$ and
$\cV_{f(x)}$. We assume that a measurable family
$\{\|\cdot\|_x\}_{x\in M}$ of norms is given which enables
us to define $\|A(x)\|:=\sup \{\|A(x)v\|_{f(x)}/\|v\|_x:
v\in\cV_x, v\neq\vec0\}$.

The Lyapunov exponents $\chi_i(x)$, and their number $k(x)$,
are measurable functions of $x$ and they are constant on orbits of
the transformation $f$. In particular, if the measure $\mu$ is
ergodic then $k$ and the $\chi_i$ are constant on a full
$\mu$-measure set of points. The subspaces $F_x^i$ also depend
measurably on the point $x$ and are invariant under the
automorphism:
$$
A(x) \cdot F_x^i = F_{f(x)}^i.
$$
If the transformation $f$ is invertible one obtains a
stronger conclusion, by applying the previous result also to
the inverse automorphism: for $n>0$ we set
$A^{-n}(x)=A(f^{-n}(x))^{-1}\cdots A(f^{-1}(x))^{-1}$ and $A^0(x)=Id$.
With this definition we associate to each $x$ a bi-infinite
sequence $(A^n(x))_{n\in\ZZ}$ of automorphisms of $\RR^d$ ($\CC^d$).
Assuming that both $\log^+\|A(x)\|$ and $\log^+\|A(x)^{-1}\|$
are in $L^1(\mu)$, one gets that there exists a
decomposition
$$
\cV_x=E_x^1\oplus\cdots\oplus E_x^k,
$$
defined at almost every point and such that $A(x) \cdot E_x^i =
E_{f(x)}^i$ and
\begin{equation}\label{eq:regular}
\lim_{n\to\pm\infty} \frac 1n \log\|A^n(x)\cdot v\| = \chi_i(x)
\end{equation}
for all $v\in E_x^i$ different from zero and all $i\in\{1,
\ldots, k\}$. Moreover the convergence in \eqref{eq:regular}
is uniform over the set $\{v\in E_x^i: \|v\|=1\}$ of unit
vectors; see \cite[Lemma 3.4.6]{BarPes2007} and \cite[Theorem
  3.5.10]{BarPes2007} for the non-invertible case.  These
\emph{Oseledets subspaces} $E_x^i$ are related to the
subspaces $F_x^i$ through
\begin{align}\label{eq:Osele-flag}
  F_x^j = \oplus_{i=1}^j E_x^i.
\end{align}
Hence, $\dim E_x^i = \dim F_x^i-\dim F_x^{i-1}$ is the
multiplicity of the Lyapunov exponent $\chi_i(x)$.

The angles between any two Oseledets subspaces decay
sub-exponentially along orbits of $f$ (see \cite[Theorem
1.3.11 \& Remark 3.1.8]{BarPes2007}):
$$
\lim_{n\to\pm\infty} \frac 1n \log\sin
\operatorname{angle}
\left(\bigoplus_{i\in I} E_{f^n(x)}^i, \bigoplus_{j\notin I} E_{f^n(x)}^j\right) = 0
$$
for any $I \subset \{1, \ldots, k\}$ and almost every point,
where for any given pair $E,F$ of complementary subspaces
(i.e. $E\oplus F=\RR^d$) we set % (this is in fact the co-sine
% of the angle and is always smaller or equal to $1$)
\begin{align*}
  \cos\ang{E,F}:=\inf\left\lbrace |\langle v,w\rangle| : \|v\|=1=\|w\|, v\in E, w\in F\right\rbrace.
\end{align*}
These facts imply the regularity condition mentioned previously and,
in particular,
\begin{equation*}%\label{multiplicative4}
\lim_{n\to\pm\infty} \frac 1n \log | \det A^n(x)| =
\sum_{i=1}^k \chi_i(x) \dim E_x^i
\end{equation*}
If the measure $\mu$ if $f$-ergodic, then the numbers
$k(x),\chi_i(x)$ are constant $\mu$-almost everywhere.  

% Consequently, if $\det A(x)=1$ at every point then the sum of all Lyapunov
% exponents, counted with multiplicity, is identically zero.

The Oseledets Multiplicative Ergodic Theorem applies, in
particular, when $f:M\to M$ is a $C^1$ diffeomorphism on
some compact manifold and $A(x)=Df_x$. Notice that the
integrability conditions are automatically satisfied for
any $f$-invariant probability measure $\mu$, since the
derivative of $f$ and its inverse are bounded in norm;
see e.g. \cite{Os68,Pe77,Man87,BarPes2007}.

%\margem{Give a reference that goes further back? Oseledets
%and Pesin!}

In general, the dependence of $E^i_x$ on $x$ is measurable.
Here we show that, under some mild assumptions on the
cocycle and the base transformation, this dependence is in
fact H{\"o}lder with probability arbitrarily close to
$1$. For the full stable or unstable subspace, H{\"o}lder
dependence on the base point has been established by
Brin~\cite{brin-app}. Our result gives H{\"o}lder dependence for each
Oseledets subspace corresponding to each Lyapunov
exponent. Furthermore, Theorem \ref{thm:KZ_global}
establishes H{\"o}lder dependence for the Kontsevich-Zorich
cocycle, while Theorem \ref{thm:KZ_disk} proves a similar
result in almost every direction on any Teichm\"uller disk.

\begin{remark}
  \label{rmk:complex-field}
  In what follows we present the statements and proofs in the
  case of real vector spaces and vector bundles, but the
  same results hold with the same proofs in the complex
  case.
\end{remark}

\subsection{Statement of the results}
\label{sec:statem-results}

Some preliminary definitions are needed.  We
define the distance from a vector $v$ to a subspace $E$ in
$\RR^d$ with a given euclidean norm $\|\cdot\|$ as
\begin{align*}
  \dist(v,E):=\min_{w\in E}\| v - w \|
\end{align*}
and then the distance between two subspaces $E,F$ of $\RR^d$
is defined to be
\begin{align}\label{eq:subspacedist}
  \dist(E,F):=\max\big\{\sup_{\|v\|=1,v\in E}\dist(v,F),
\sup_{\|v\|=1,v\in F}\dist(v,E)\big\}.
\end{align}

Let $M$ be a metric space endowed with a distance $d$.
We say that the cocycle $A:M\to GL(d,\RR)$ is
$\nu$-H\"older-continuous if  there exists a constant $H>0$
such that
\begin{align*}
  \|A(x)-A(y)\|\le H d(x,y)^\nu, \quad\text{for all $x,y\in M$}.
\end{align*}
In this case we also say that $A$ is a $(H,\nu)$-H\"older
cocycle.  Note that we do not assume that the base is
compact; however, the H\"older assumption is uniform on the
base.

% \subsubsection{The case of linear cocycles over
%   (non-invertible) transformations}
% \label{sec:case-linear-cocycl}

\begin{maintheorem}
  \label{mthm:HolderBundles-ae}
  Let $f:M\circlearrowleft$ be a Lipschitz transformation
  preserving an ergodic probability measure $\mu$, and
  $A:M\to GL(d,\RR)$ a $\nu$-H\"older cocycle over $f$ such
  that both $\log\|A(x)\|$ and $\log\|A(x)^{-1}\|$ are
  $\mu$-integrable. We denote by $\chi_1<\dots<\chi_k$ the
  distinct $k$ Lyapunov exponents associated to this
  cocycle, and by $F^i_x$ the filtration by subspaces of
  $\RR^N$ corresponding to each $\chi_i$, $i=1,\dots,k$,
  defined for $\mu$-almost every $x\in M$.

  Then, for every $\epsilon>0$, there exists a compact
  subset $\Lambda_\epsilon$ of $M$, and constants
  $C=C(\Lambda_\epsilon)>0$,
  $\omega_i=\omega_i(\chi_1,\dots,\chi_k)\in(0,1),
  i=1,\dots,k$ and
  $\delta=\delta(\epsilon,A,\Lambda_\epsilon,\chi_1,\dots,\chi_k)>0$,
  such that
  \begin{enumerate}
  \item $\mu(\Lambda_\epsilon)\ge1-\epsilon$;
  \item for all $x,y\in\Lambda_\epsilon$ with
    $d(x,y)<\delta$ we have $ \dist(F^i_x,F^i_y)\le
    C_\epsilon d(x,y)^{\nu\omega_i}.$
  \end{enumerate}
  Moreover, if the map $f$ is also invertible, then letting
  $E^i_x$ be the Oseledets subspaces corresponding to each
  $\chi_i$, $i=1,\dots,k$, which form a splitting
  $\RR^n=E^1_x\oplus \dots\oplus E^k_x$ for
    $\mu$-almost every $x\in M$, we also have
    \begin{enumerate}
    \item[(3)] $ \dist(E^i_x,E^i_y)\le C_\epsilon
      d(x,y)^{\nu\omega_i}$ for all
      $x,y\in\Lambda_\epsilon$ with $d(x,y)<\delta$.
\end{enumerate}
\end{maintheorem}

We say that a measurable subbundle $E=E(x)$ of $\RR^d$
defined over a measurable subset $\Lambda$ is locally
$\alpha$-H\"older continuous with H\"older constant $H>0$
if, for every $x\in\Lambda$, there exists $\delta>0$ such
that
\begin{align*}
  \dist(E(x),E(y))\le H d(x,y)^\alpha
  \quad\text{for all}\quad
  x,y\in\Lambda \quad\text{with}\quad d(x,y)<\delta.
\end{align*}

\begin{remark}
  \label{rmk:deltax}
  The definition of H\"older-continuity for measurable
  subbundles allows the local radius $\delta>0$ to depend on
  $x$. This is not the case for linear cocycles, but is
  needed for vector bundle automorphisms; see
  Theorem~\ref{mthm:Holder-bundles-vectorbundle} in the
  following Subsection~\ref{sec:case-vector-bundle-1}.
\end{remark}

Thus Theorem~\ref{mthm:HolderBundles-ae} means that
the subbundles given by Oseledets' Theorem are
H\"older-continuous on arbitrarily big (in measure) compact
subsets of the ambient space, for H\"older cocycles over a
Lipschitz base transformation.

\begin{remark}
  \label{rmk:C-infty}
  We have $C_\epsilon=C(\Lambda_\epsilon)\to+\infty$ and
  $\delta\to0$ as $\epsilon\to0+$, so that the H\"older
  constant becomes worse as the size of $\Lambda_\epsilon$
  grows to fill a full $\mu$-measure subset of $M$.
\end{remark}

The proof shows that the dependence of the Oseledets
directions on the base point at every block
$\Lambda_\epsilon$ is H\"older, with H\"older exponent
only dependent on the minimum gap
$\epsilon_0=\min\{\chi_{i+i}-\chi_i : 1\le i < k\}$ of the
Lyapunov spectrum, and H\"older constant essentially
dependent on the choice of the regular block.

\subsubsection{The case of linear multiplicative cocycles over flows}
\label{sec:case-linear-multipl}

We can extend our result to the case of linear
multiplicative cocycles over flows with the exact same
conclusions (1), (2) and (3) of
Theorem~\ref{mthm:HolderBundles-ae} as a corollary of the
proof for invertible transformations.

We recall some common
definitions. Let $(X,\mu)$ be a Lebesgue space.
The measurable map $\vfi:\RR\times X\to X$ is a
\emph{measurable flow} if $\vfi_0=Id$ and
$\vfi_t\circ\vfi_s=\vfi_{t+s}$ for each $t,s\in\RR$. The
flow $\vfi$ preserves the measure $\mu$ if
$\vfi_t:=\vfi(t,\cdot)$ preserves $\mu$ for all $t\in\RR$.

A measurable function $A:X\times\RR\to GL(d,\RR)$ is a
\emph{linear multiplicative cocycle over $\vfi$} if
$A(x,0)=Id$ and $A(x,t+s)=A(\vfi_t(x),s)A(x,t)$ for all
$s,t\in\RR$ and for every $x\in X$.

We can also consider a measurable linear cocycle $A$ acting
on a vector bundle $\cV$ over $M$, that is
$A(x,t):\cV_x\to\cV_{\vfi_t(x)}$ is a linear bijection for
all $(x,t)\in M\times\RR$. Taking a family of norms on the
vector bundle enables us to define the norm of $A(x,t)$.

Let us take a measurable cocycle $A$ over a flow $\vfi$
which preserves a probability measure $\mu$ such that
\begin{align*}
  \sup_{-1\le t\le1}\log^+\|A(x,t)\|\in L^1(X,\mu).
\end{align*}
Oseledets' Theorem ensures that for $\mu$-almost every $x\in M$
there exists $k=k(x)\in\{1, \ldots, d\}$, a filtration $
\{0\}=F^0_x\subset F^1_x \subset \cdots \subset F_x^{k-1}
\subset F_x^k=\cV_x, $ and numbers
$\chi_1(x)<\cdots<\chi_k(x)$ such that
\begin{align*}
  F^i_x&=\{v\in\cV_x: \lim_{t\to\infty} \frac 1t
  \log\|A(x,t)\cdot v\| \le \chi_i(x)\}, \qand
  \\
  \lim_{t\to\infty} &\frac 1t \log\|A(x,t)\cdot v\| =
  \chi_i(x) \quad\text{for all $v\in F_x^i\setminus
    F_x^{i-1}$ and $i\in\{1,\ldots, k\}$}.
\end{align*}
The subspaces are invariant: $A(x,t)\cdot F^i_x =
F^i_{\vfi_t(x)}$ and depend measurably on the base point
$x\in X$.  The function $k(x)$ and the Lyapunov exponents
$\chi_i(x)$ are measurable functions, constant on orbits of
the flow $\vfi$ and so, if $\mu$ is ergodic, these are
constant functions almost everywhere. Moreover there exists
a decomposition $\cV_x=E_x^1\oplus\cdots\oplus E_x^k$
defined $\mu$-almost everywhere satisfying $A(x,t) \cdot
E_x^i = E_{\vfi_t(x)}^i$ and
\begin{equation}\label{eq:flowregular}
\lim_{t\to\pm\infty} \frac 1t \log\|A(x,t)\cdot v\| =
\chi_i(x) = \chi_i,
\quad 0\neq v\in E_i(x), \quad i=1,\dots,k.
\end{equation}
As before the convergence in
\eqref{eq:flowregular} is uniform over the set of unit
vectors and the Oseledets subspaces $E_x^i$ are related to
the subspaces $F_x^j$ as follows $$F_x^j = \oplus_{i=1}^j E_x^i.$$

Now let a linear multiplicative cocycle $A(x,t)$ over a
flow $\vfi$ preserving a probability measure $\mu$ which is
also ergodic with respect to $\vfi$ be given. That is, we
assume that every measurable subset $A$ of $X$ satisfying
$\vfi_t(A)=A$ for all $t\in\RR$ also satisfies
$\mu(A)\cdot\mu(X\setminus A)=0$. Then it is well known that
there exists a denumerable subset $Y$ of $\RR$ such that the
bijection $\vfi_t$ is ergodic with respect to $\mu$ for all
$t\in\RR\setminus Y$; see e.g. \cite{PuSh71}.

Let us fix $0<\tau\in\RR\setminus Y$. Then $A^\tau:X\to
GL(d,\RR), x\in X\mapsto A(x,\tau)$ defines a cocycle over
$\vfi_\tau$ since $A(x,0)=Id$ and
\begin{align*}
  A(x,n\tau)
  &=
  A(\vfi_{(n-1)\tau}(x),\tau)\cdots 
  A(\vfi_{\tau}(x),1) A(x,\tau), \quad n\ge0;
  \\
  A(x,n\tau)
  &=
  A(\vfi_{-n\tau}(x),-\tau)\cdots
  A(\vfi_{-\tau}(x),-\tau)
  ,\quad n<0
\end{align*}
where $A(\cdot,-\tau)=A(\cdot,\tau)^{-1}$ by the cocycle
property.

We now note that by ergodicity, since the following limits
exist, we have the equalities
\begin{align*}
  \lim_{t\to\pm\infty}\frac1t\log\|A(x,t)v\| 
  =
  \chi_i
  =
  \lim_{n\to\pm\infty} \frac1n\log\|A(x,n\tau)v\|,
  \quad
  v\in E_i(x),\, \mu-\text{a.e.  } x.
\end{align*}
Therefore the Lyapunov exponents and Oseledets subspaces of
the cocycle $A^\tau$ over the invertible transformation
$f=\vfi_\tau$ coincide with those of the cocycle $A$ over
the flow $\vfi$.

We obtain the following result by a direct application of
Theorem~\ref{mthm:HolderBundles-ae} to the cocycle $A^\tau$
and the transformation $f$. We say that a linear
multiplicative cocycle $A:X\times\RR\to X$ over a flow
$\vfi$ on the metric space $X$ (which is a Lebesgue space
with the Borel $\sigma$-algebra) is $\nu$-H\"older if, for
every $t\in\RR$ there exists a constant $H=H_t>0$ such that
\begin{align*}
  \|A(x,t)-A(y,t)\|\le H d(x,y)^\nu, \quad x,y\in X.
\end{align*}

\begin{maincorollary}
  \label{mcor:Holder-Osel-flows}
  Let $\vfi:\RR\times X\to X$ be a flow preserving an
  ergodic probability measure $\mu$ and such that
  $\vfi_t:X\to X$ is a Lipschitz transformation for all
  $t\in\RR$. Let also $A:X\times\RR\to GL(d,\RR)$ be a
  $\nu$-H\"older linear multiplicative cocycle over $\vfi$
  such that $\sup_{t\in[-1,1]}\log^+\|A(x,t)\|$ is a
  $\mu$-integrable function, denote by $\chi_1<\dots<\chi_k$
  the $k$ distinct Lyapunov exponents associated to this
  cocycle, and by $E^i_x$ the Oseledets subspaces
  corresponding to each $\chi_i$, $i=1,\dots,k$, which form
  a splitting $\RR^n=E^1_x\oplus \dots\oplus E^k_x$ for
  $\mu$-almost every $x\in X$.
 
  Then, for every $\epsilon>0$, there exists a compact
  subset $\Lambda_\epsilon$ of $X$, and constants
  $C=C(\Lambda_\epsilon)>0$,
  $\omega_i=\omega_i(\chi_1,\dots,\chi_k)\in(0,1),
  i=1,\dots,k$ and
  $\delta=\delta(\epsilon,A,\Lambda_\epsilon,\chi_1,\dots,\chi_k)>0$,
  such that
  \begin{enumerate}
  \item $\mu(\Lambda_\epsilon)\ge1-\epsilon$;
  \item $ \dist(E^i_x,E^i_y)\le C_\epsilon
    d(x,y)^{\nu\omega_i}$ for all $x,y\in\Lambda_\epsilon$
    with $d(x,y)<\delta$.
\end{enumerate}
\end{maincorollary}

\begin{remark}\label{rmk:rescale-time}
  It is no restriction to assume that $\tau=1$ above since
  we can always make a linear rescale of time, e.g., we may
  set $s=\tau \cdot t$ as a new time variable.
\end{remark}

\subsubsection{The case of vector bundle automorphisms}
\label{sec:case-vector-bundle-1}

Now $A:\cV\to\cV$ is an automorphism of the $d$-dimensional
vector bundle $\cV$ covering $f:M\to M$. We assume that $f$
is Lipschitz with Lipschitz constant $L>0$, that $M$ is a
finite $m$-dimensional manifold and that $A$ is \emph{locally
  H\"older}. This means that 
\begin{itemize}
\item we can find an at most denumerable locally finite open
  cover $(U_i,\psi_i)_{i\ge1}$ of $M$ together with
  trivializing charts of the bundle $\cV$, that is,
  $\psi_i:p^{-1}(U_i)\to U_i\times \RR^d$ given by
  $w\in\cV_x\mapsto(x,\psi_{i,x}(w))$ and
  $\psi_{i,x}:\cV_x\to\RR^d$ a linear bijection for all $
  x\in U_i$, whose overlaps satisfy
  \begin{align*}
    \psi_{j,i}:=\psi_i\circ\psi_j^{-1}:(U_i\cap
    U_j)\times\RR^d\to (U_i\cap U_j)\times\RR^d, \quad
    (x,v)\mapsto
    (x,\psi_{j,i,x}(v))=(x,\psi_{i,x}\circ\psi_{j,x}^{-1}(v))
  \end{align*}
  for $x\in U_i\cap U_j$ and $x\in U_i\cap U_j\mapsto
  \psi_{j,i,x}\in GL(d,\RR)$;
\item we assume without loss of generality that the open
  sets of the cover of $M$ are also domains of charts
  $\phi_i:U_i\to\RR^m$ of $M$; so we can define on $U_i$ a
  distance $d_i(x,y)=\|\phi_i(x)-\phi_i(y)\|_2$ for
  $x,y\in U_i$ and all possible $i$, where $\|\cdot\|_2$ is
  the Euclidean norm in $\RR^m$;
\item since the initial open cover was locally
finite, then given any point $x\in M$ we have at most
finitely many chart domains $U_{i_1(x)},\dots, U_{i_k(x)}$
containing $x$. Hence we can define
\begin{align*}
  d_x(y):=\min\{d_{i_1(x)}(y,x), \dots,
  d_{i_k(x)}(y,x)\}
  \quad\text{for all}\quad y\in
  U_{i_1(x)}\cap\dots\cap U_{i_k(x)}.
\end{align*}
In this way, we write for $x\in M$ and $\xi>0$
\begin{align*}
  B(x,\xi):=\{y\in U_{i_1(x)}\cap\dots\cap U_{i_k(x)} : d_x(y)<\xi\}.
\end{align*}

\item likewise we define the norm $\|v\|_x:=
\sup\{\|\psi_{i_j(x),x}(v)\|_2: j=1,\dots,k\}$ for each
$x\in M$, where $\|\cdot\|_2$ is the Euclidean norm in
$\RR^d$. This is a measurable family of norms on the bundle
$\cV$ which is locally constant. Naturally we write
$\|A(x)\|$ for $\sup \{\|A(x)v\|_{f(x)}/\|v\|_x: v\in\cV_x,
v\neq\vec0\}$ in what follows;

\item for any given $x$ and $i,j$ such that $x\in U_j$,
  $f(x)\in U_i$ we assume that $A(x)$ is $(H,\nu)$-H\"older
  as before on a neighborhood of $x$:
\begin{align*}
  \|A(x)-A(y)\|\le H\cdot
  d_x(y)^\nu, \quad\text{for all
    $x,y\in U_j\cap f^{-1}(U_i)$}
\end{align*}
where we allow $H=H_{j,i}$ to depend on $i,j$.
\end{itemize}

The proof of the previous
Theorem~\ref{mthm:HolderBundles-ae} uses Lemma~\ref{le:Mane}
which says that there exists $c_1>0$ such that the H\"older
constant of an iterate $A^n$ of an H\"older cocycle $A$ is
bounded by $c_1^n$ for all $n\ge1$.  We cannot identify
naturally all the subspaces $\cV_x$ where $A(x)$ acts, so we
assume that iterates $A^n$ of the vector bundle automorphism
have H\"older constant similarly bounded in the following
way.

We take an ergodic $f$-invariant probability measure $\mu$
and the family $(U_i)_{i\ge1}$ of open charts of the vector
bundle, and assume that
\begin{itemize}
\item the partition $\cP$ given by the intersection of the
  these sets is a $\mu-\bmod0$ partition of $M$;
\item there exists a full $\mu$-measure subset $Y$ of $M$
  such that the refined partitions $\cP_n:=\cP\vee
  f^{-1}\cP\vee\dots\vee f^{-n+1}\cP$ \emph{do not shrink
    faster than exponentially}, that is, they satisfy: there
  exists $0<\xi<1$ such that for all $x\in Y$ there is
  $c=c(x)>0$ so that
  \begin{align}\label{eq:innerball}
    B(x,c\xi^n)\subset\cP_n(x) \quad\text{for all}\quad
    n\ge1.
  \end{align}
In the above statement, as usually, $\cP(x)$ denotes the
atom of $\cP$ which contains $x$; or $\cP(x)=\emptyset$ if
$x$ is not contained in any atom (which can only happen for
a zero $\mu$-measure subset of points).

\item the diameter of $\cP_n$ can be made sufficiently small
  $\mu$-almost everywhere, that is, there exists
  $\delta_0=\delta_0(\epsilon,A,\Lambda_\epsilon,\chi_1,\dots,\chi_k)>0$
  and for $x\in Y$ there exists $n$ such that
  $\diam\cP_n(x)<\delta_0$.

\item there exists $c_1>0$ such that for all $n\ge1$, $x\in
  Y$ and $y\in\cP_n(x)$
  \begin{align}\label{eq:Holder-bundle-automorph}
    \|A^n(x)-A^n(y)\|\le c_1^n d_x(y)^\nu.
  \end{align}
\end{itemize}
Above we implicitly assume that the
bound~\eqref{eq:Holder-bundle-automorph} does not depend on
the choices of $i,j$.

We say that a vector bundle automorphism $A$ satisfying the
above properties is an \emph{admissible $(c_1,\nu)$-H\"older
  vector bundle automorphism with respect to $\mu$.} With
these notions we can now state our main result.

\begin{maintheorem}
  \label{mthm:Holder-bundles-vectorbundle}
  Let $A$ be an admissible $(c_1,\nu)$-H\"older vector
  bundle automorphism with respect to the ergodic
  $f$-invariant probability measure $\mu$ such that both
  $\log\|A(x)\|$ and $\log\|A(x)^{-1}\|$ are
  $\mu$-integrable. We denote by $\chi_1<\dots<\chi_k$ the
  $k$ distinct Lyapunov exponents associated to this
  cocycle, and by $F^i_x$ the filtration by subspaces of
  $\RR^N$ corresponding to each $\chi_i$, $i=1,\dots,k$,
  defined for $\mu$-almost every $x\in M$.

  Then, for every $\epsilon>0$, there exists a compact
  subset $\Lambda_\epsilon$ of $M$, and constants
  $C=C(\Lambda_\epsilon)>0$,
  $\omega_i=\omega_i(\chi_1,\dots,\chi_k)\in(0,1),
  i=1,\dots,k$ % and
  % $\delta=\delta(\epsilon,A,\Lambda_\epsilon,\chi_1,\dots,\chi_k)>0$,
  such that
  \begin{enumerate}
  \item $\mu(\Lambda_\epsilon)\ge1-\epsilon$;
  \item for all $x\in\Lambda_\epsilon$ there exists
    $\delta=\delta(x)>0$
    % $N=N(x)\in\ZZ^+$
    such that for all %$n\ge N$ and
    %$y\in\cP_n(x)\cap\Lambda_\epsilon$ 
    $y\in B(x,\delta)$ we have $
    \dist(F^i_x,F^i_y)\le C_\epsilon
    d_x(y)^{\nu\omega_i}.$
  \end{enumerate}
  Moreover, if the map $f$ is also invertible, then letting
  $E^i_x$ be the Oseledets subspaces corresponding to each
  $\chi_i$, $i=1,\dots,k$, which form a splitting
  $T_x M=E^1_x\oplus \dots\oplus E^k_x$ for $\mu$-almost
  every $x\in M$, we also have $\delta=\delta(x)>0$ % $N=N(x)\in\ZZ^+$
  defined for
  $x\in\Lambda_\epsilon$ satisfying
    \begin{enumerate}
    \item[(3)] $ \dist(E^i_x,E^i_y)\le C_\epsilon
      d_x(y)^{\nu\omega_i}$ for all $x\in\Lambda_\epsilon$,
      $y\in \Lambda_\epsilon\cap B(x,\delta)$.
      %$n\ge N$ and $y\in\cP_n(x)\cap\Lambda_\epsilon$.
\end{enumerate}
\end{maintheorem}

\begin{remark}
  \label{rmk:admissible-cocycle}
  Condition~\eqref{eq:innerball} can be easily obtained in
  many examples as follows.  
  It is well-known that, if
  $\theta(x):=-\log\dist_x(\partial\cP(x),x)$ is
  $\mu$-integrable, then we have $(1/n)\cdot\theta\circ
  f^n\xrightarrow[n\to+\infty]{}0, \mu-$a.e.  Hence given
  $\alpha>0$ for $\mu$-a.e. $x$ we can find $N(x)>1$ such
  that $\dist_x(\partial\cP(f^n(x)), f^n(x)) \ge e^{-\alpha
    n}$ for all $n\ge N(x)$. Therefore there exists $c(x)>0$
  such that $\dist_x(\partial\cP(f^n(x)), f^n(x)) \ge
  c(x)e^{-\alpha n}$ for all $n\ge1$.
  
  If in addition we assume that $f$ is locally Lipschitz,
  then we can find $0<\xi<1$ so that
  assumption~\eqref{eq:innerball} is satisfied.

  Thus $\mu$-integrability of $\theta$ together with a
  Lipschitz condition on $f$ is enough to show that the
  atoms of the refined partitions $\cP_n$ do not shrink at a
  rate faster than exponential.
\end{remark}

\subsubsection{The case of a vector bundle automorphism
  covering a flow}
\label{sec:case-vector-bundle-2}

We may, analogously to the case of a vector bundle
automorphism, consider a vector bundle automorphism covering
a flow. We provide the relevant definitions and just note
that the proof is a straightforward corollary of the
previous Theorem~\ref{mthm:Holder-bundles-vectorbundle}.

The measurable function $A:\cV\times\RR\to\cV$ on a vector
bundle $\cV$ is a linear multiplicative cocycle covering a flow
$\vfi:\RR\times X\to X$ on the base manifold $X$ of $\cV$ if
$A(x,t):\cV_x\to\cV_{\vfi_t(x)}$ is a linear isomorphism
satisfying the cocycle property
\begin{align*}
  A(x,0)=Id:\cV_x\to\cV_x
  \qand
  A(x,t+s)=A(\vfi_t(x),s)\cdot A(x,t),
  \quad x\in X, s,t\in\RR.
\end{align*}

\begin{maincorollary}
  \label{mcor:multcocycleflow}
  Let $A$ be a linear multiplicative cocycle on a vector
  bundle covering a flow $\vfi:\RR\times M\to M$ which
  preserves an ergodic probability measure $\mu$ such that
  $\sup_{t\in[-1,1]}\log^+\|A(x,t)\|$ is a $\mu$-integrable
  function. We assume that the transformation $f=\vfi_1$ is
  Lipschitz; that $\mu$ is also $f$-ergodic and that
  $A^1:=A(\cdot,1)$ is an admissible $(c_1,\nu)$-H\"older
  vector bundle automorphism with respect to $f$ and $\mu$.
  We denote by $\chi_1<\dots<\chi_k$ the $k$ distinct
  Lyapunov exponents associated to the cocycle $A$ and by
  $E^i_x$ the Oseledets subspaces corresponding to each
  $\chi_i$, $i=1,\dots,k$, which form a splitting $T_x
  M=E^1_x\oplus \dots\oplus E^k_x$ for $\mu$-almost every
  $x\in M$.

  Then, for every $\epsilon>0$, there exists a compact
  subset $\Lambda_\epsilon$ of $M$, and constants
  $C=C(\Lambda_\epsilon)>0$,
  $\omega_i=\omega_i(\chi_1,\dots,\chi_k)\in(0,1),
  i=1,\dots,k$ such that
  \begin{enumerate}
  \item $\mu(\Lambda_\epsilon)\ge1-\epsilon$;
  \item for all $x\in\Lambda_\epsilon$ there exists $\delta=\delta(x)>0$
    %$N=N(x)\in\ZZ^+$ 
    such that for all %$n\ge N$ and
    %$y\in\cP_n(x)\cap\Lambda_\epsilon$ 
    $y\in\Lambda_\epsilon\cap B(x,\delta)$ we have $
    \dist(E^i_x,E^i_y)\le C_\epsilon
    d_x(y)^{\nu\omega_i}.$
  \end{enumerate}
\end{maincorollary}

\subsection{The Kontsevich-Zorich cocycle}

Conditions~\eqref{eq:innerball} and
\eqref{eq:Holder-bundle-automorph} are non-trivial to verify
in the case of the Kontsevich-Zorich cocycle (see
\cite{Zorich_survey} for a survey).  We thus deal with the
corresponding theorem separately in Section \ref{sec:KZ}
(see Theorem \ref{thm:KZ_global}).  The result is as
follows.
  
\begin{maintheorem}
  Let $\cM$ be an affine invariant manifold and let $\mu$ be
  the corresponding ergodic $\SL_2\bR$-invariant probability
  measure (see \cite{EM}).  Let $E$ be the Kontsevich-Zorich
  cocycle (or any of its tensor powers) and let
  $\{\lambda_i\}$ be its Lyapunov exponents, with Oseledets
  subspaces $E^{i}$.
 
  Then there exists $\nu_i>0$ such that for any $\epsilon>0$
  there exists a compact set $K_\epsilon$ with
  $\mu(K_\epsilon)>1-\epsilon$ and such that the spaces
  $E^i$ vary $\nu_i$-H\"older continuously on $K_\epsilon$.
\end{maintheorem}
Similarly in spirit to a recent result of Chaika-Eskin
\cite{Chaika-Eskin}, the above result in fact generalizes to
individual Teichm\"uller disks (see Theorem
\ref{thm:KZ_disk}).
\begin{maintheorem}
  Let $x\in \cH(\kappa)$ be a flat surface in a stratum and
  let $E$ be the Kontsevich-Zorich cocycle.  If $E$ is some
  tensor power of the KZ cocycle, let $k\geq 1$ be the
  smallest order of the tensor product which contains it.
 
  For $g\in \SL_2\bR$ denote by $A(g,x):E_x \to E_{gx}$ the
  matrix of the cocycle $E$.  By Theorem 1.2 in
  \cite{Chaika-Eskin} there exist $\lambda_1>\cdots
  >\lambda_k$ such that for a.e. $\theta\in S^1$ we
  have \begin{enumerate}
 \item There exists a decomposition $E_x = \oplus_i
   E^{i}(\theta)_x$ with $E^{i}(\theta)_x$ measurably
   varying in $\theta$
  \item We have (where $g_t^\theta$ is the geodesic flow in direction $\theta$)
  \begin{align*}
    \lim_{t\to \pm \infty} \frac 1 t \log
    \norm{A(g_t^\theta, x) v_i} = \lambda_i \hskip 0.5cm
    \forall v_i\in E^i(\theta)_x
  \end{align*}
 \end{enumerate}
 Let $\nu_i:=\frac 1{2k} \min \left(\log \frac
   {\lambda_i}{\lambda_{i+1}}, \log \frac
   {\lambda_{i-1}}{\lambda_i} \right) - \epsilon'$, for
 arbitrarily small $\epsilon'$.
 
 Then for any $\epsilon>0$ there exists a compact set
 $K_\epsilon\subseteq S^1$ of Lebesgue measure at least
 $1-\epsilon$ such that the subspaces $E_i(\theta)$ vary
 $\nu_i$-H\"older continuously for $\theta\in K_\epsilon$.
\end{maintheorem}

\section{Preliminary results and definitions}
\label{sec:prelim-results-defin}

In what follows we adopt the convention that for a point
$x\in M$ its iterates are denoted by $x_i:=f^i(x)$ for
$i\in\ZZ$ if $f$ is invertible. For non-invertible $f$ the
meaning of $x_i$ with $i<0$ is given by the set of $i$th
pre-images $x_i=(f^i)^{-1}(\{x\})$.

\subsection{Regular blocks for (possibly) non-invertible
  base transformation}
\label{sec:regular-blocks-non}

Using the definition of Lyapunov exponents and of the
subspaces of the filtration $F^i_x$ and of the
Oseledets-Pesin Reduction Theorem in \cite[Theorem
3.5.5]{BarPes2007}, we can find sets $\Lambda^\ell_\epsilon$
where the expansion/contraction rates are bounded, as
follows.

For simplicity we assume that $\mu$ is $f$-ergodic and so
function $k$, the exponents $\chi_1,\dots,\chi_k$ and the
dimensions $\dim F^i_x-\dim F^{i-1}_x$ for $i=1,\dots,k$ are
constant on a full $\mu$-measure subset.
We denote by $\epsilon_0=\min\{\chi_{i+i}-\chi_i :
1\le i < k\}$ the minimum gap in the Lyapunov spectrum.

For any given $\ell\in\ZZ^+$ and $0<\epsilon<\epsilon_0/10$
we define the \emph{regular block} $\Lambda_\epsilon^\ell$
as the set of points $x\in M$ such that for $i=1,\dots,k$
and all $m,n\ge0$
\begin{align*}
  \|A^n(x_m) v_i\| 
  &\le
  \ell e^{(\chi_i(x)+\epsilon)n+\epsilon m} \|v_i\|, \quad
  \forall v_i\in
  A^m(x)F^i_x;
  \\
  \ell^{-1} e^{(\chi_i(x)-\epsilon)n+\epsilon m} \|v_i\|
  &\le 
  \|A^n(x_m) v_i\| 
  \le
  \ell e^{(\chi_i(x)+\epsilon)n+\epsilon m} \|v_i\|, 
  \;\;\forall v_i\in \big(A^m(x)F^{i-1}_x\big)^\perp\cap A^m(x)
  F^i_x,
\end{align*}
where $x_m=f^m(x)$ and $(\cdot)^\perp$ denotes the
orthogonal complement with respect to the inner product that
defines $\|\cdot\|$.

\subsection{Regular blocks for invertible base transformation}
\label{sec:regular-sets}

Again using the definition of Lyapunov exponents and
Oseledets subspaces, we can find sets
$\Lambda^\ell_\epsilon$ where the expansion/contraction
rates are uniform and the angles between the subspaces are
bounded away from zero, as follows.

We assume that $\mu$ is $f$-ergodic and so the function
$k$, the exponents $\chi_1,\dots,\chi_k$ and the dimensions
$\dim E^i_x$ for $i=1,\dots,k$ are constant on a full
$\mu$-measure subset. 

We continue to write $\epsilon_0>0$ for the minimum gap in
the Lyapunov spectrum and for $\ell\in\ZZ^+$ and
$0<\epsilon<\epsilon_0/10$ we define $\Lambda_\epsilon^\ell$
as the set of points $x\in M$ such that for $i=1,\dots,k$
and for all $m\in\ZZ$, writing $x_m=f^m(x)$ we have
\begin{align*}
  \ell^{-1} e^{(\chi_i(x)-\epsilon)n-\epsilon |m|} \|v_i\| 
  &\le 
  \|A^n(x_m) v_i\| 
  \le
  \ell e^{(\chi_i(x)+\epsilon)n+\epsilon|m|} \|v_i\|, \quad \forall v_i\in
  A^m(x)E_i(x), \quad \forall n\ge0;
  \\
  \ell^{-1} e^{(\chi_i(x)+\epsilon)n-\epsilon|m|} \|v_i\|
  &\le 
  \|A^n(x_m) v_i\| 
  \le
  \ell e^{(\chi_i(x)-\epsilon)n+\epsilon|m|} \|v_i\|, 
  \quad \forall v_i\in A^m(x)E_i(x), \quad \forall n\le0;\text{ and}
  \\
  \cos\operatorname{angle}\left( \bigoplus_{i\in I} E_{f^n(x)}^i\right. &, \left.\bigoplus_{j\notin
      I} E_{f^n(x)}^j \right)
  \le 1-\frac{e^{-\epsilon|n|}}{\ell} \text{ for any } I
  \subset \{1, \ldots, k\} \text{ and for all } n\in\ZZ.
\end{align*}

The control of the norm along the subspaces
$F^i_x$ on $\Lambda^\ell_\epsilon$ implies easily that there
exists $L>0$ such that
\begin{align}\label{eq:expboundnorm}
  \|A^n(x_m)\|\le \ell L^{|n|}e^{\epsilon m}
  \quad\text{for all $m,n\in\ZZ$ and every $x\in\Lambda^\ell_\epsilon$}
\end{align}
with $L=e^{2(\chi_k-\chi_1)}\le\exp\left(4\int\log\|A(x)\|\,d\mu(x)\right)$.

From the Theorem of Oseledets we have that for any given
small $\epsilon>0$ we can find $\ell_0\in\ZZ^+$ such that
$\mu(\Lambda_\epsilon^\ell)>0$ for $\ell\ge\ell_0$ and
$\cup_{\ell\ge\ell_0} \Lambda_\epsilon^\ell$ has full
$\mu$-measure in $M$. We note that to have
$\mu(\Lambda_\epsilon)>1-\epsilon$ with smaller $\epsilon>0$
we must increase the value of $\ell\in\ZZ^+$.

\subsection{Consequences of the angle control on regular
  blocks for invertible base}
\label{sec:conseq-angle-control}

We remark that the angle condition has the following
consequence: given a pair of complementary subspaces $E,F$
with $\ang{E,F}\le 1-1/\ell$ for some $\ell\in\ZZ^+$, then
for any vector $v+w$ with $v\in E, w\in F$ we have
$\|v\|\le \ell \|v+w\|$ and $\|w\|\le\ell\|v+w\|$. 

Indeed, we consider $E\oplus F$ with the given norm
$\|\cdot\|$ and $E\times F$ with the norm
$|(v,w)|:=(\|v\|^2+\|w\|^2)^{1/2}$ for $(v,w)\in E\times
F$. Then we can write for any $v\in E$ and $w\in F$ with
$v+w\neq0$ that
\begin{align*}
  v&=\lambda v_0 \qand w=\mu w_0 \text{  with  } \lambda,\mu\in\RR^+,
  \|v_0\|=1=\|w_0\|; \qand
  \\
  \frac{\|v+w\|^2}{|v+w|^2}
  &=
  \frac{\|v_0+\mu w_0/\lambda\|^2}{|v_0+\mu w_0/\lambda|^2}
  =
  \frac{1+2(\mu/\lambda)<v_0,w_0>+(\mu/\lambda)^2}{1+(\mu/\lambda)^2}
  =
  1+\frac{2\lambda\mu}{\lambda^2+\mu^2}<v_0,w_0>,
\end{align*}
where $<v_0,w_0>\in[-1,1]$ and $0\le 2\lambda\mu/(\lambda^2+\mu^2)\le1$.
Hence we get
\begin{align*}
  \|v\| = |v| &\le |v+w| =
  \left(1+\frac{2\lambda\mu}{\lambda^2+\mu^2}<v_0,w_0>
  \right)^{-1} \!\!\!\!\|v+w\|
  \le
  \frac1{1-\ang{E,F}} \|v+w\|
  \le \ell\|v+w\|.
\end{align*}

% \subsection{Lyapunov exponents and ergodic sums}
% \label{sec:lyapun-exponents-erg}

% Given an invertible $f:M\circlearrowleft$ and $A:M\to
% GL(d,\RR)$ as before, we induce the measurable map $\hat f:
% M\times \sS^d \to M\times \sS^d$ by
% \begin{align*}
%   \hat f(x,u)=\Big(f(x),\frac{A(x)\cdot u}{\big\| A(x)\cdot u  \big\|}\Big)
% \end{align*}
% We denote $\hat M:=M\times\sS^d$ and
% let $\psi:\hat M\to\RR$ be given by $\psi(x,v):=\log\|A(x)\cdot
% v\|$.

% %  and, for each $f$-invariant $\eta$ let
% % $\chi_\eta:=\int\psi\,d\eta$. 

% If $x$ is a point where Lyapunov exponents are defined (an
% ``Oseledets point'') and $v\in E_x^i$,
% then we have the following telescopic sum
%   \begin{align*}
%     \lim_{n\to+\infty}\frac1n S_n\psi(x,v)
%     &=
%     \lim_{n\to+\infty}\frac1n\sum_{j=0}^{n-1} \psi(\hat
%     f^j(x,v))
%     \\
%     &=
%     \lim_{n\to+\infty}\frac1n\sum_{j=0}^{n-1}
%     \log\left\|
%       A\big(f^j(x)\big)\cdot \frac{A^j(x)\cdot
%         v}{\|A^j(x)\cdot v\|}
%       \right\|
%       \\
%       &=
%       \lim_{n\to+\infty}\frac1n\sum_{j=0}^{n-1}
%       \log\frac{\|A^{j+1}(x)\cdot v\|}{\|A^j(x)\cdot v\|}
%       \\
%       &=
%       \lim_{n\to+\infty}\frac1n
%       \log\|A^n(x)\cdot v\|
%       =\chi_i(x).
%   \end{align*}
% The Multiplicative Ergodic Theorem ensures that the sets
% \begin{align*}
%   X_i(N,\delta):=\left\{x\in M: \forall v\in E_x^i, \|v\|=1,
%   \,  \left|\chi_i(x)-\frac1n S_n\psi(x,v)\right| \le \delta,
%   \forall n\ge N\right\}
% \end{align*}
% satisfy $\mu(M\setminus X_i(N,\epsilon))\xrightarrow[N\to+\infty]{}0$
% for every $1\le i<k(x)$ and $\delta>0$.

\subsection{H\"older estimates for exponentially
  expanded/contracted subspaces }
\label{sec:locally-constant-ose}

Now we need the following technical results whose proofs we
postpone until Section~\ref{sec:lemmata}.

We compare the definition given in~\eqref{eq:subspacedist}
to an alternative description of the distance between
subspaces $E,F$ of $\RR^d$. We can obtain a linear map $L:
E\to E^\perp$ such that its graph $\{u+Lu:u\in E\}$ equals
$F$ and obtain that, on the one hand
\begin{align*}
  \sup_{\|u+Lu\|=1,u\in E}& \dist(u+Lu, E)
  =
  \sup_{\|u\|^2+\|Lu\|^2=1, u\in E} \min_{w\in E}\|u+Lu-w\|
  =
  \sup_{\|u\|^2+\|Lu\|^2=1, u\in E} \|Lu\|
  \\
  &=
  \sup\{ \|L\|\cdot \|u_0\| : u_0\in E \text{   s.t.  }
  \|u_0\|^2+\|Lu_0\|^2=1 \text{  and  }
  \|L\|=\frac{\|Lu_0\|}{\|u_0\|}\}
  =
  \frac{\|L\|}{\sqrt{1+\|L\|^2}}.
\end{align*}
On the other hand, for $u_1\in E$ such that $\|u_1\|=1$ and
$\|Lu_1\|=\|L\|$
\begin{align*}
  \sup_{\|u\|=1,u\in E}\dist(u,F)
  &=
  \sup_{\|u\|=1,u\in E} \min_{v\in E} \|u-(v+Lv)\|
  \le
  \|u_1-(u_1-Lu_1)\| = \|Lu_1\| \le \|L\|.
\end{align*}
Therefore from the definition of $\dist(E,F)$ we get
\begin{align}
  \label{eq:normdist}
  \frac{\|L\|}{\sqrt{1+\|L\|^2}} \le \dist(E,F) \le \|L\|
\end{align}
and we may estimate (or indeed define) the distance between
$E$ and $F$ by $\|L\|$.

% Now we take $M$ a compact submanifold of $\RR^N$ and $f$ a
% $C^{1+\alpha}$ diffeomorphism with a globally defined pair of
% complementary fiber bundles $E\oplus F$. 

The following simple lemma enables us to provide a rough
control of the distance between subspaces which are
exponentially expanded/contracted by a pair of sequences of
linear maps, such that the norm of the difference between
these maps is a sequence that grows at most exponentially
fast.

  \begin{lemma}(from \cite[Lemma 5.3.3]{BarPes2007})
  \label{le:seqmatrixesHolder}
  Let $(A_n)_{n\ge1},(B_n)_{n\ge1}$ be two sequences of real
  $N\times N$ matrices such that for some $0<\lambda<\mu$
  and $C\ge1$ there exist subspaces $E,E^\prime,F,F^\prime$ of
  $\RR^N$ satisfying $\RR^N=E\oplus E^\prime=F\oplus
  F^\prime$ with $d>0$ such that
  \begin{align}\label{eq:bddangle}
    u=v+w, v\in E, w\in E^\prime \text{  or  } v\in F, w\in
    F^\prime \implies \max\{\|v\|,\|w\|\}\le d\|u\|,
  \end{align}
  and, for some fixed $n>0$
  \begin{itemize}
  \item $\|A_n u\|\le C\lambda^n\|u\|$ for $u\in E$; and
    $C^{-1}\mu^n \|v\|\le \|A_n v\|$ for
    $v\in E^\prime$;
  \item $\|B_n u\|\le C\lambda^n \|u\|$ for $u\in F$; and
  $C^{-1}\mu^n \|v\|\le\|B_n v\|$ for $v\in
    F^{\prime}$.
  \end{itemize}
  Then for every pair
  $(\delta,a)\in(0,1]\times[\lambda,+\infty)$ satisfying
  \begin{align}\label{eq:HipHold}
    \left(\frac{\lambda}{a}\right)^{n+1}<\delta
    % \le
    % \left(\frac{\lambda}{a}\right)^{n}
    \quad\text{and}\quad
    \|A_n-B_n\|\le \delta a^n
  \end{align}
  we get
  $\dist(E,F)\le (2+d)C^2\frac{\mu}{\lambda}
    \delta^{\log(\mu/\lambda)/\log(a/\lambda)}$.
%  and
    % $\dist(E^\prime, F^\prime)\le
    % d(C^{-1}-\lambda/\mu)^{-1}\delta^{\log(\sigma/\mu)/\log (\lambda/a)}$.
\end{lemma}

In the previous lemma we assumed condition
(\ref{eq:bddangle}) instead of taking $E'=E^\perp$ and
$F'=F^\perp$ as in \cite[Lemma 5.3.3]{BarPes2007}.

Now we state a result which enables us to show that the
iterated cocycle $A^n$, of a $\nu$-H\"older cocycle $A$ over
a Lipschitz base map $f$, is also $\nu$-H\"older for every
$n\ge1$. Moreover the H\"older constant of $A^n$ grows at
most exponentially fast with $n$.

\begin{lemma}
  \label{le:Mane}
  Let us assume that $A:M\to GL(d,\RR)$ is $\nu$-H\"older
  with constant $c_0=c_0(A,\nu)$ and there exists $L>1$ such
  that $f:M\circlearrowleft$ is Lipschitz with constant $L$
  and $\|A^n(x)\|\le L^{|n|}$ for all $x$ is some fixed
  compact $\Lambda\subset M$ and $n\in\ZZ^+$.  Then for
  $c_1=\max\{e^\epsilon,L^{1+\nu},1+c_0\}$ we have
  $\|A^n(x)-A^n(y)\|\le c_1^n \dist(x,y)^\nu$ for all
  $x,y\in \Lambda$ and $n\in\ZZ^+$.
\end{lemma}

The previous lemmas combined show (see e.g. \cite[Theorem
5.3.1]{BarPes2007}) that the splitting of the tangent space
into the directions with positive and negative Lyapunov exponents
depends H\"older continuously on the base point over every
``regular block'' associated to an
invariant probability measure with non-zero Lyapunov
exponents.

We want to extend this conclusion to all the Oseledets
subspaces for any ergodic invariant probability. For that we
need to study splittings of the tangent bundle into three
subbundles exhibiting \emph{distinct exponential growth}
under the action of a sequence of linear operators, as we
state below. This allows us to analyze the behavior of all
Oseledets subspaces by associating them into a splitting
with three subbundles in different combinations. Otherwise
Lemma~\ref{le:seqmatrixesHolder} would only allow the study
of the flags $
\{0\}=F^0_x\subset F^1_x \subset \cdots \subset F_x^{k-1}
\subset F_x^k=\RR^d$ for $\mu$-a.e. $x\in M$.

The condition of distinct exponential growth rates
encompasses the case of zero Lyapunov exponents in one
statement.

The main lemma we need is the following extension of
Lemma~\ref{le:seqmatrixesHolder}.

\begin{lemma}
  \label{le:3bundlesmatrixHolder}
  Let $(A_n)_{n\in\ZZ},(B_n)_{n\in\ZZ}$ be two bi-infinite
  sequences of real $N\times N$ invertible matrices
  admitting
  $0<\lambda_1<\lambda_2<\mu_1<\mu_2<\sigma_1<\sigma_2$ and
  $C,d>0$ such that there exist subspaces $E^*,F^*, G^*$ of
  $\RR^N$ satisfying for all $n>0$
    \begin{itemize}
    \item 
      $\RR^N=E^*\oplus F^*\oplus G^*$ for $*=A,B$ and
      \begin{itemize}
      \item for $u\in E^A$: $C\lambda_1^n\|u\|\le\|A_n
        u\|\le C\lambda_2^n\|u\|$ and
        $C\lambda_2^{-n}\|u\|\le\|A_{-n} u\|\le
        C\lambda_1^{-n}\|u\|$;
      \item for $v\in F^A$: $C\mu_1^n\|v\|\le\|A_n v\|\le
        C\mu_2^n\|v\|$ and $C\mu_2^{-n}\|v\|\le\|A_{-n}
        v\|\le C\mu_1^{-n}\|v\|$;
      \item for $w\in G^A$: $C\sigma_1^n\|w\|\le\|A_n w\|
        \le C \sigma_2^{n} \|w\|$ and
        $C\sigma_2^{-n}\|w\|\le\|A_{-n} w\|\le
        C\sigma_1^{-n}\|w\|$.
      \item for $u\in E^B$: $C\lambda_1^n\|u\|\le\|B_n
        u\|\le C\lambda_2^n\|u\|$ and
        $C\lambda_2^{-n}\|u\|\le\|B_{-n} u\|\le
        C\lambda_1^{-n}\|u\|$;
      \item for $v\in F^B$: $C\mu_1^n\|v\|\le\|B_n v\|\le
        C\mu_2^n\|v\|$ and $C\mu_2^{-n}\|v\|\le\|B_{-n}
        v\|\le C\mu_1^{-n}\|v\|$;
      \item for $w\in G^B$: $C\sigma_1^n\|w\|\le\|B_n w\|
        \le C \sigma_2^{n} \|w\|$ and
        $C\sigma_2^{-n}\|w\|\le\|B_{-n} w\|\le
        C\sigma_1^{-n}\|w\|$.
      \end{itemize}
    \item for $*=A,B$, if $u=v+w$ and either
        \begin{align*}
          v\in E^*&, v\in F^*\oplus G^* \quad\text{or}\quad
          v\in F^*, w\in G^* \quad\text{or}
          \\
          v\in E^*\oplus F^*&, w\in G^* \quad\text{or}\quad
          v\in E^*, w\in F^*,
      \end{align*}
      then $\|v\|\le d\|u\|$.
    \end{itemize}
    Then there exists $a>\lambda_2+1/\lambda_2+\sigma_1$ and
    $\delta_0=\delta_0(a,C,\mu_1/\lambda_2,\sigma_1/\mu_2)
    \in (0,1)$ such that if, for some $n>0$, we have
  \begin{align}\label{eq:boundnorm}
    \|A_n\|\le a^{n} \quad\text{and}\quad  \|A_{-n}\|\le a^{n}
  \end{align}
  and for some $0<\delta<\delta_0$, we also have
  \begin{align}\label{eq:boundist}
    \|A_n-B_n\|\le \delta a^{n} \quad\text{and}\quad
    \|A_{-n}-B_{-n}\|\le \delta a^{n}
  \end{align}
  then the following relations are true:
  \begin{align*}
    \dist(E^A,E^B)&\le (2+d)C^2\frac{\mu_1}{\lambda_2}
    \delta^\alpha,\qquad
\dist(F^A,F^B)\le\frac92 (2+3d)^{1+\eta}C^{2(1+\eta)}
    \frac{\sigma_1\mu_1^{\eta}}{\mu_2\lambda_2^{\eta}}
    \delta^{\beta}\qand
\\
\dist(G^A,G^B) &\le
(2+d)C^2\frac{\sigma_1}{\mu_2}\delta^\gamma, 
% \qand
% \dist(F^A\oplus G^A,F^B\oplus G^B)\le (2+d)C^2\frac{\mu_1}{\lambda_2} \delta^\omega,
  \end{align*}
  where $\alpha:=\log(\mu_1/\lambda_2)/\log(a/\lambda_2)$,
  $\eta=\log(\sigma_1/\mu_2)/\log(a/\mu_2)$,
%$\omega:=\log(\mu_1/\lambda_2)/\log(a\mu_1)$,
  $\gamma:=\log(\sigma_1/\mu_2)/\log(a\sigma_1)$ and also
  $\beta=[\log(\mu_1/\lambda_2)\log(\sigma_1/\mu_2)]
  /[\log(a\mu_1)\log(a/\mu_2)]$.
\end{lemma}

  We remark that $\alpha,\beta$ and $\gamma$ are positive and
  smaller than $1$ by the choice of $a$ and the order
  relations between the rates of expansion/contraction.

%%%%%%%%%%%%%%%%%%%%%%%%%%%%%%%%%%%%%%%%%%%%%%%%%%%%%%%

  \section{H\"older continuity of the Oseledets splitting on
    regular blocks}
\label{sec:holder-contin-oseled}

Here we prove the main theorem by combining the lemmas from
Section~\ref{sec:prelim-results-defin} and applying them to
regular sets of a (not necessarily invertible) co-cycle $A$
over $f$.  After that we prove the invertible case of the
main theorem. Then we prove the lemmas in
Section~\ref{sec:lemmata}.

We assume the first two conditions on Lemma~\ref{le:Mane},
i.e. $f:M\circlearrowleft$ is Lipschitz and $A:M\to
GL(\RR,d)$ is $(c_0,\nu)$-H\"older; and check the other
condition, that is, that $\|A^n(x)\|$ grows at most
exponentially for $n\ge1$. We continue under the assumption
that $\mu$ is $f$-ergodic.

\subsection{The case of non-invertible base}
\label{sec:case-non-invert}

For $\ell\ge\ell_0$ and $0<\epsilon<\epsilon_0$ we set
$\Lambda_\epsilon=\Lambda^\ell_\epsilon$ such that
$\mu(\Lambda_\epsilon)>1-\epsilon$, from
Section~\ref{sec:regular-blocks-non}, we take $x,y\in
\Lambda_\epsilon$ and for each $1<i<k$, $n\in\ZZ^+$ we
define
\begin{align*}
  A_n&=A^n(x), \quad B_n=A^n(y) \qand L=e^{2(\chi_k-\chi_1)}
  \qand C=\ell;
  \\
  \lambda
  &=e^{\chi_i+\epsilon},
  \mu=e^{\chi_{i+1}-\epsilon},
  \sigma=e^{\chi_k+\epsilon} \qand d=1;
  \\
  E&=F^{i}_x, \quad
  E^\prime= \big(F^i_x\big)^\perp\qand
  F=F^{i}_y, \quad
  F^\prime= \big(F^{i}_y \big)^\perp.
\end{align*}
From the definition of regular block in the possibly
non-invertible case, in
Section~\ref{sec:regular-blocks-non}, we have guaranteed the
upper bound on the exponential growth of $\|A^n(x)\|$ for
$x\in\Lambda_\epsilon^\ell$ and $n\ge0$.

The value $d=1$ comes from the choice of $E^\prime,F^\prime$
as orthogonal complements with respect to the inner product
that defines $\|\cdot\|$.

From Lemma~\ref{le:Mane} we can find a constant $a>c_1$ big
enough such that the condition~\eqref{eq:HipHold} in
Lemma~\ref{le:seqmatrixesHolder} is true for the choices of
sequences $A_n,B_n$ and for some $n\in\ZZ^+$ with
$0<\delta=\dist(x,y)^\nu<\min\{\chi_1/c_1\}$. Here $\chi_1$
is the smallest Lyapunov exponent and $c_1$ is given by
Lemma~\ref{le:Mane}.

Indeed, for $\delta$ and $\lambda$ chosen as above for any
given $i=2,\dots,k-1$, since $\chi_1\le\lambda$ we can
certainly find $n>0$ such that
$(\lambda/a)^{n+1}<\delta$ and then proceed
using the statements of Lemmas~\ref{le:seqmatrixesHolder}
and \ref{le:Mane}.  The other assumptions on
Lemma~\ref{le:seqmatrixesHolder} are true by
the definition of hyperbolic block in
Section~\ref{sec:regular-blocks-non} together with the
choices of subspaces above.
%  (the choice of $\sigma$ is
% unimportant in this context).

We can then apply Lemma~\ref{le:seqmatrixesHolder} to
conclude that for $i=1,\dots, k-1$
\begin{align*}
  \dist(F^i_x,F^i_y)
  \le
  3\ell^2 e^{\eta_i}
  \dist(x,y)^{\nu\omega_i}
  \quad\text{where}\quad
  \eta_i=\chi_{i+1}-\chi_i-2\epsilon
  \qand
  \omega_i=\frac{\eta_i}{\log a  -\chi_i-\epsilon}.
\end{align*}
% Here $\ell$ grows without bound when $\epsilon\to0+$.
This completes the proof of
Theorem~\ref{mthm:HolderBundles-ae} in the non-invertible
case once we set $C_\epsilon=\ell^2\max_{i=1,\dots,k-1}
\{e^{\eta_i}\}$.

\subsection{The case of invertible base}
  \label{sec:case-invert-base}

  Again for $\ell\ge\ell_0$ and $0<\epsilon<\epsilon_0$ we
  set $\Lambda_\epsilon=\Lambda^\ell_\epsilon$ such that
  $\mu(\Lambda_\epsilon)>1-\epsilon$, from
  Section~\ref{sec:regular-sets}, we take $x,y\in
  \Lambda_\epsilon$ and for each $1<i<k$, $n\in\ZZ$ we
  define
\begin{align*}
  A_n&=A^n(x), \quad B_n=A^n(y) \qand L=e^{2(\chi_k-\chi_1)}
  \qand C=d=\ell;
  \\
  \lambda_1&=e^{\chi_1-\epsilon},
  \lambda_2=e^{\chi_{i-1}+\epsilon},
  \mu_1=e^{\chi_i-\epsilon},
  \mu_2=e^{\chi_i+\epsilon},
  \sigma_1=e^{\chi_{i+1}-\epsilon},
  \sigma_2=e^{\chi_k+\epsilon};
  \\
  E^A&=F^{i-1}_x=\oplus_{j=1}^{i-1}E^j_x, \quad
  F^A=E^i_x \qand
  G^A=\oplus_{j=i+1}^k E^i_x
  \\
  E^B&=F^{i-1}_y=\oplus_{j=1}^{i-1}E^j_y, \quad
  F^B=E^i_y \qand
  G^B=\oplus_{j=i+1}^k E^i_y.
\end{align*}
From the definition of regular block, in
Section~\ref{sec:regular-sets}, we know that the assumption
on exponential growth of $\|A^n(x)\|$ for
$x\in\Lambda_\epsilon^\ell$ and $n\in\ZZ$ is guaranteed; see
the observation in~\eqref{eq:expboundnorm}.  So from
Lemma~\ref{le:Mane} we can find a constant $a>0$ big enough
such that condition~\eqref{eq:boundnorm} and item (1) in the
statement of Lemma~\ref{le:3bundlesmatrixHolder} are true
for the choices of sequences $A_n,B_n$ and subspaces above
with $0<\delta=\dist(x,y)^\nu<\min\{\chi_1/c_1,\delta_0\}$
and some $n\in\ZZ^+$. By the definition of the
regular block $\Lambda_\epsilon^\ell$, and the choices of
constants, all the conditions of item (2) and (3) are
satisfied also.  The statement of
Lemma~\ref{le:3bundlesmatrixHolder} ensures then that
\begin{align*}
  \dist(E^i_x,E^i_y)\le C_\ell
  e^{\chi_{i+1}-\chi_{i}-2\epsilon+\eta(\chi_{i}-\chi_{i-1}-2\epsilon)} d(x,y)^{\nu\beta} \quad
  \text{where}\quad
  \eta=\frac{\chi_{i+1}-\chi_{i}-2\epsilon}{\log a-\chi_i-\epsilon},
\end{align*}
$\beta=\eta\frac{\chi_i-\chi_{i-1}-2\epsilon}{\chi_i-\epsilon+\log
  a}$ and $C_\ell\to\infty$ monotonically as
$\ell\to+\infty$.

For the subbundle $E^1$ (with smallest Lyapunov exponent) we
take $i=2$ as above and consider the estimate for
$\dist(E^A,E^B)$ from Lemma~\ref{le:3bundlesmatrixHolder} to
obtain
\begin{align*}
  \dist(E^1_x,E^1_y)\le C_\ell
  e^{\chi_{2}-\chi_{1}-2\epsilon} d(x,y)^{\nu\alpha} \quad
  \text{where}\quad
  \alpha=\frac{\chi_2-\chi_1-2\epsilon}{\log a - \chi_1
    -\epsilon}.
\end{align*}
Finally for the subbundle $E^k$ (with largest Lyapunov
exponent) we take $i=k-1$ as above and consider the estimate
for $\dist(G^A,G^B)$ from
Lemma~\ref{le:3bundlesmatrixHolder} to obtain
\begin{align*}
  \dist(E^k_x,E^k_y)\le C_\ell
  e^{\chi_{k}-\chi_{k-1}-2\epsilon} d(x,y)^{\nu\gamma} \quad
  \text{where}\quad
  \gamma=\frac{\chi_k-\chi_{k-1}-2\epsilon}{\log a + \chi_k
    -\epsilon}.
\end{align*}

This shows that the dependence of the Oseledets directions
on the base point at every regular block is H\"older, with
H\"older exponent essentially dependent on the minimum gap
$\epsilon_0$ of the Lyapunov spectrum, and H\"older constant
essentially dependent on the choice of the regular
block. The proof of Theorem~\ref{mthm:HolderBundles-ae} is
complete once we set $C_\epsilon$ to equal the maximum of
the factors multiplying $\dist(x,y)$ in the above
expressions.

\subsection{The case of vector bundle automorphisms}
\label{sec:case-vector-bundle}

Now $A:\cV\to\cV$ is an automorphism of the $d$-dimensional
vector bundle $\cV$ covering the map $f:M\to M$ on the
$m$- dimensional manifold $M$. We assume that $f$ admits
an ergodic $f$-invariant probability such that $A$ is
admissible with respect to $\mu$.

We again start by choosing a regular block
$\Lambda_\epsilon=\Lambda_\epsilon^\ell$ for some
$0<\epsilon<\epsilon_0$ and $\ell\ge\ell_0$, such that
$\mu(\Lambda_\epsilon)>1-\epsilon$.
We recall that for $\mu$-a.e. $x$ we have defined
in~\eqref{eq:innerball} a positive function $c(x)$ relating
the distance between $x$ and $\cP_n(x)$. For $\mu$-a.e.
$x$ let $N_1=N(x)\in\ZZ^+$ be the first
integer so that $c(x)>1/2^N$. Then we have
\begin{align}\label{eq:boundaway}
  n>N_1\implies y\in\cP_n(x)\setminus\cP_{n+1}(x)
  \quad\text{satisfies}\quad
  \dist_x(y)\ge c(x)\xi^{n+1} > \left(\frac{\xi}2\right)^{n+1}.
\end{align}
Let $N_2=N_2(x)\ge N_1$ be such that $\diam\cP_n(x)<\chi_1/c_1$ for
all $n>N_2$.

\subsubsection{The case of non-invertible base}
\label{sec:case-non-invert-1}

With the previous choices, for each $1<i<k$ and $n\in\ZZ^+$
we define $x_n=f^n(x)$ and $y_n=f^n(y)$ for
$y\in\Lambda_\epsilon\cap\cP_{N_2}(x)$
\begin{align}
  A_n&=\psi_{j_n,x_n}A^n(x)\psi_{i,x}^{-1}, \text{ with  }
  x_n\in U_{j_n}, x\in U_i; \label{eq:An-global}
  \\
  B_n&=\psi_{r_n,y_n}A^n(y)\psi_{i,y}^{-1}, \text{ with  }
  y_n\in U_{r_n}, x\in U_i;\label{eq:Bn-global}
\end{align}
and then choose $L, C, \lambda, \mu, \sigma, d, E, E^\prime,
F, F^\prime$ and also $\delta=d_x(y)^\nu$ as in
Section~\ref{sec:case-non-invert}. 

There exists $n\ge N_2$ such that
$y\in\cP_{n}(x)\setminus\cP_{n+1}(x)$, thus if we choose
$a>c_1$ big enough, since $\lambda\ge\chi_1$, we have
$(\lambda/a)^{n+1}< (\xi/2)^{(n+1)/\nu}<d_x(y)$. For this it
is enough to take $a>c_1+\chi_k\cdot(\xi/2)^{-\nu}$.  This,
together with the definition of hyperbolic block, ensures
that the conditions of Lemma~\ref{le:seqmatrixesHolder} are
verified for $y\in\Lambda_\epsilon\cap\cP_{N_2}(x)$.
Finally we note that, from~\eqref{eq:innerball}, every point
$y\in \Lambda_\epsilon\cap B(x, c(x)\xi^{N_2})$ belongs to
$\Lambda_\epsilon\cap\cP_{N_2}(x)$, so that we arrive at the same
conclusion as in Section~\ref{sec:case-non-invert}.

This completes the proof of
Theorem~\ref{mthm:Holder-bundles-vectorbundle} in the
non-invertible case, if we set $\delta(x):=c(x)\xi^{N_2}$.

\subsubsection{The case of invertible base}
\label{sec:case-invert-base-1}

We again use the previous choices of $N_2=N_2(x)$ and, for
each $1<i<k$, we set $L, C, \lambda_1, \lambda_2, \mu_1, \mu_2,
\sigma_1, \sigma_2$ as in
Section~\ref{sec:case-invert-base}. This provides a value
$\delta_0=\delta_0(i)$ from
Lemma~\ref{le:3bundlesmatrixHolder}. We now choose
$\kappa_i=\kappa_i(x)\ge N_2$ so that
$\diam\cP_{n}(x)<\delta_0(i)$ for all $n>\kappa_i$.

For each $n\in\ZZ$, we define $x_n=f^n(x)$ and $y_n=f^n(y)$
for $y\in\Lambda_\epsilon\cap\cP_{\kappa_i}(x)$, and we make
the same definitions as in~\eqref{eq:An-global}
and~\eqref{eq:Bn-global}.

We then take $E^A, F^A, G^A, E^B, F^B, G^B$ as in
Section~\ref{sec:case-invert-base}.  From the definition of
hyperbolic block in the invertible case we can verify the
conditions of Lemma~\ref{le:3bundlesmatrixHolder} with
$\delta=d_x(y)^\nu$ for some positive integer $n$. For
that we have to choose $a$ sufficiently large as explained
above.

So we obtain the same conclusions as in
Section~\ref{sec:case-invert-base} for each $1<i<k$ and also
for $i=1$ and $i=k$. We then set
$N=N(x)=\max\{\kappa_1,\dots,\kappa_k\}$ and obtain the
conclusion of Theorem~\ref{mthm:Holder-bundles-vectorbundle}
for $y\in\Lambda_\epsilon\cap\cP_n(x)$ with $n\ge N$.

To conclude, since $\lambda_\epsilon\cap B(x,c(x)\xi^{N})$
is contained in $\lambda_\epsilon\cap\cP_{N}(x)$, setting
$\delta(x)=c(x)\xi^{N}$ completes the proof of
Theorem~\ref{mthm:Holder-bundles-vectorbundle}.

%%%%%%%%%%%%%%%%%%%%%%%%%%%%%%%%%%%%%%%%%%%%%%%%%%%%%%%

\section{Proofs of the Lemmata}
\label{sec:lemmata}

Now we present a proof of the technical lemmata from
Section~\ref{sec:prelim-results-defin}.

\begin{proof}[Proof of Lemma~\ref{le:seqmatrixesHolder}] 
The proof follows the one given in \cite[Lemma
5.3.3]{BarPes2007} closely.

  Let us fix $n>0$ as in the statement.  We define the
  following cones $Q=\{u\in\RR^N:\|A_n u\|\le
  2C\lambda^n\|u\|\}$ and $R=\{u\in\RR^N:\|B_n u\|\le
  2C\lambda^n\|u\|\}$. We decompose $u\in\RR^N$ into $v+w$
  with $v\in E$ and $w\in E^\prime$. If $u\in Q$, then since
  $\|v\|\le d\|u\|$ by assumption we see that
  \begin{align*}
   2C\lambda^n\|u\|\ge \|A_n u\|= \|A_n(v+w)\|\ge
    C^{-1}\mu^n\|w\|-C\lambda^n\|v\|
    \ge C^{-1}\mu^n\|w\|-C\lambda^n d \|u\|
  \end{align*}
  implies
  $\|w\|\le(2+d)C^2\left(\frac{\lambda}{\mu}\right)^n\|u\|$. Therefore
  $\dist(u,E)\le
  (2+d)C^2\left(\frac{\lambda}{\mu}\right)^n\|u\|$ for all $u\in
  Q$.
  
  Now for $a>\lambda$ and $\delta\in(0,1]$ such that
  \eqref{eq:HipHold} is true we set
  $\gamma:=\lambda/a\in(0,1)$ and observe that if $u\in F$
  then
  \begin{align*}
    \|A_n u\| 
    &\le 
    \|B_n u\| + \|A_n - B_n\|\|u\|
    \le
    C\lambda^n\|u\| + \delta a^n \|u\|
    \\
    &\le
    (C\lambda^n + (\gamma a)^n) \|u\|
    \le
    2C\lambda^n\|u\|
  \end{align*}
  and so $u\in Q$, that is, $F\subset Q$. Symmetrically we
  get $E\subset R$. Hence using again~\eqref{eq:HipHold} we
  conclude that $\dist(E,F)\le (2+d)C^2(\lambda/\mu)^n \le
  (2+d)C^2\frac{\lambda}{\mu}\delta^{\log(\lambda/\mu)/\log(\lambda/a)}$.

  % For the distance between the complementary subspaces, the
  % argument is analogous. Let us fix again $n>0$ as in
  % the statement together with the following pair of cones:
  % $Q^\prime=\{u\in\RR^N:
  % \|A_n(u)\|\ge(C^{-1}-\lambda/\mu)\mu^n\|u\|\}$ and
  % $R^\prime=\{ u\in\RR^N:
  % \|B_n(u)\|\ge(C^{-1}-\lambda/\mu)\mu^n\|u\|\}$. Then for
  % $u=v+w\in Q^\prime$ with $v\in E$ and $w\in E^\prime$ we
  % get
  % \begin{align*}
  %   \|A_n(u)\|
  %   &\ge
  %   (C^{-1}-\lambda/\mu)\mu^n\|u\|
  %   \ge 
  %   (C^{-1}-\lambda/\mu)\mu^n\frac{\|w\|}d
  %   \quad\text{so that}
  %   \\
  %   \|w\|
  %   &\le
  %   d\left(\frac1C-\frac{\lambda}{\mu}\right)^{-1}
  %   \left(\frac{\sigma}{\mu}\right)^n  \|u\|
  %   \quad\text{for all $u\in Q^\prime$}.
  % \end{align*}
  % So we obtain $\dist(u,E^\prime)\le
  % d(C^{-1}-\lambda/\mu)^{-1}
  % (\sigma/\mu)^n \|u\|, u\in Q^\prime$. Finally, for $u\in
  % F^\prime$ we have by the choice of $a,\delta$
  % \begin{align*}
  %   \|A_n(u)\|
  %   &\ge
  %   \|B_n(u)\|-\|A_n-B_n\|\|u\|
  %   \ge
  %   C^{-1}\mu^n-\delta a^n\|u\|
  %   \ge
  %   (C^{-1}\mu^n-\lambda^n)\|u\|
  %   \\
  %   &=
  %   \mu^n\left(\frac1C-\Big(\frac{\lambda}{\mu}\Big)^n\right)\|u\|
  %   \ge
  %   \mu^n\left(\frac1C-\frac{\lambda}{\mu}\right)\|u\|
  % \end{align*}
  % and so $F^\prime\subset Q^\prime$. As in the first part of
  % the proof, symmetrically we get $E^\prime\subset
  % R^\prime$, thus $\dist(F^\prime, E^\prime)\le
  % d(C^{-1}-\lambda/\mu)^{-1} (\sigma/\mu)^n \le
  % d(C^{-1}-\lambda/\mu)^{-1}
  % \delta^{\log(\sigma/\mu)/\log(\lambda/a)}.$
\end{proof}

\begin{proof}[Proof of Lemma~\ref{le:Mane}]
  We follow~\cite[proof of Lemma 13.5]{Man87} (see
  also~\cite[Lemma 5.3.4]{BarPes2007}) and argue by
  induction on $n\in\NN$ and write $x_n=f^n(x_0)$ and
  $y_n=f^n(y_0)$ for $n\ge0$ and $x_0,y_0\in\Lambda$. For
  $n=1$ we have $\|A(x_0)-A(y_0)\|\le c_0
  \dist(x_0,y_0)^\nu$ by the H\"older assumption, for some
  $c_0>0$.

  Let us assume that we can find $c>c_0$ as in the statement
  of the lemma for $k=1,\dots,n$ and let us see what we need
  to extend the property for $k=n+1$. % We write
  % $z_n=f^n(z_0)$ for all $z_0\in M$ and $n\ge0$ in what
  % follows.
Since we have a cocycle with values in $GL(d,\RR)$, we can
write
  \begin{align*}
    A^{n+1}(x_0)-A^{n+1}(y_0)
    =
    A(x_n)\cdot(A^n(x_0)-A^n(y_0))+(A(x_n)-A(y_n))\cdot A^n(y_0)
  \end{align*}
  adding and subtracting $A(x_n)\cdot A^n(y_0)$.  We recall
  that $\log L>0$ is assumed to be an upper bound for
  $\{n^{-1}\log \|A^n(z)\|: z\in M,n\in\NN\}$ and write,
  using the induction assumption
  \begin{align*}
    \|A^{n+1}(x_0)-A^{n+1}(y_0)\|
    &\le
    \|A(x_n)\|\cdot\|A^n(x_0)-A^n(y_0)\|
    +
    \|A^n(y_0)\|\cdot\|A(x_n)-A(y_n)\|
    \\
    &\le
    e^{\epsilon n} C^n \dist(x_0,y_0)^\nu + L^n c_0 \dist(x_n,y_n)^\nu.
  \end{align*}
Now since $L$ is also a Lipschitz constant of $f$ we get
$\dist(x_n,y_n)\le L^n\dist(x_0,y_0)$ and so we can bound
the previous expression
\begin{align*}
  \le [( e^\epsilon C)^n + L^n c_0 L^{n\nu} ]  \dist (x_0,y_0)^\nu
  = [  (e^\epsilon C)^n + c_0 L^{n(1+\nu)} ]  \dist (x_0,y_0)^\nu.
\end{align*}
To complete the inductive step all we need is that
\begin{align*}
  C^{n+1} \ge  (C e^\epsilon)^n + c_0 L^{n(1+\nu)}
  \quad\text{that is}\quad
  C
  \ge 
  \left(\frac{e^\epsilon}{C}\right)^n 
  + c_0 \left( \frac{L^{1+\nu}}{C}\right)^n
  \quad\text{for all}\quad n\ge0
\end{align*}
which can easily be achieved by taking a sufficiently large
$C\ge c_1>c_0>0$ as in the statement of the lemma.
\end{proof}

\begin{proof}[Proof of Lemma~\ref{le:3bundlesmatrixHolder}]
  We use Lemma~\ref{le:seqmatrixesHolder}
  applied to the splitting $E^*\oplus (F^*\oplus G^*)$ to deduce
  that, on the one hand
  \begin{align}\label{eq:alfa}
    \dist(E^A,E^B) \le (2+d)C^2\frac{\mu_1}{\lambda_2}
    \delta^{\log(\mu_1/\lambda_2)/\log(a/\lambda_2)}
    = (2+d)C^2\frac{\mu_1}{\lambda_2}
    \delta^\alpha
  \end{align}
  and, on the other hand
  \begin{align}\label{eq:beta}
    \dist(V^A,V^B)\le
    (2+d)C^2\frac{\lambda_2^{-1}}{\mu_1^{-1}}
    \delta^{\log(\lambda_2^{-1}/\mu_1^{-1})/\log(a/\mu_1^{-1})}
    =
    \widetilde C\frac{\mu_1}{\lambda_2} 
    \delta^{\log(\mu_1/\lambda_2)/\log(a\mu_1)}
    =\widetilde C\frac{\mu_1}{\lambda_2} 
    \delta^\omega,
  \end{align}
  where $\widetilde C=(2+d)C^2$, $V^A:=F^A\oplus G^A$ and
  $V^B:=F^B\oplus G^B$ and
  $\omega=\log(\mu_1/\lambda_2)/\log(a\mu_1)$ for
  simplicity. The bound for the distance between $V^A,
  V^B$ is deduced using Lemma~\ref{le:seqmatrixesHolder}
  applied to the sequences $(A_{-n})_{n\ge0}$ and
  $(B_{-n})_{n\ge0}$.  The hypothesis in the statement of
  Lemma~\ref{le:3bundlesmatrixHolder} exactly provide the
  necessary conditions to apply
  Lemma~\ref{le:seqmatrixesHolder} to these sequences of
  matrices.

  We can also apply Lemma~\ref{le:seqmatrixesHolder} for the
  sequences $(A_{-n})_{n\ge0}$ and $(B_{-n})_{n\ge0}$ but
  now with the splittings $U^A\oplus G^A$ and $U^B\oplus
  G^B$ with $U^A:=E^A\oplus F^A$ and $U^B:=E^B\oplus
  F^B$. We obtain
  \begin{align}\label{eq:gama}
    \dist(G^A,G^B) \le \widetilde C \frac{\mu_2^{-1}}{\sigma_1^{-1}}
    \delta^{\log(\mu_2^{-1}/\sigma_1^{-1})/\log(a/\sigma_1^{-1})}
    =
    \widetilde C\frac{\sigma_1}{\mu_2}\delta^{\log(\sigma_1/\mu_2)/\log(a\sigma_1)}
    =\widetilde C\frac{\sigma_1}{\mu_2}\delta^\gamma.
  \end{align}

  So we can find a linear map $L:V^A\to (V^A)^\perp$
  such that its graph equals $V^B$, that is
  \begin{align}\label{eq:hatF1AB}
    V^B=\{ u+Lu : u\in V^A \}
    \qand
    \frac{\|L\|}{\sqrt{1+\|L\|^2}} \le \dist(V^A,V^B) \le \|L\|
  \end{align}
  by the estimate obtained in \eqref{eq:normdist}. 

  This allows us to identify $V^A$ with $V^B$ via the
  isomorphism 
  \begin{align*}
    \phi:=I+L : V^A \to V^B, u \mapsto u+Lu
  \end{align*}
  and then pass the splitting $F^B\oplus G^B$ of $V^B$ to
  a corresponding splitting $\hat F^B\oplus \hat G^B$ of
  $V^A$ with $\hat F^B=\phi^{-1} F^B, \hat G^B=
  \phi^{-1}G^B$. 

  Now we can compare the distance between the pair
  $F^A,\hat F^B$ using the action of the sequences of matrices
  $A_n$ and $\hat B_n=\phi^{-1} B_n \phi$ and
  Lemma~\ref{le:seqmatrixesHolder} again. To finish we need
  to estimate
  \begin{align}\label{eq:triangle}
    \dist(F^A,F^B)\le \dist(F^A,\hat F^B)+\dist(\hat F^B,F^B).
  \end{align}
  We first consider the estimation of $\dist(F^A,\hat F^B)$.
  We already have the right conditions for the action of
  $A_n$ over the splitting $F^A\oplus G^A$. As for $\hat
  B_n$ it is easy to see that
  \begin{align*}
    \frac1{\|\phi\|} \|B_n\mid F^B\|
    \le
    \|\hat B_n\mid \hat F^B\|=\|\phi^{-1}B_n\mid F^B\|
    \le
    \|\phi^{-1}\| \cdot \|B_n\mid F^B\|
  \end{align*}
  and since the image of $L$ is orthogonal to the domain, we
  have $\|\phi(v)\|\ge\|v\|$ for all $v$, thus
  \begin{align*}
   \|\phi^{-1}\|\le1 % \|(I+L)^{-1}\|\le\frac1{1-\|L\|} 
    \qand
    \|\phi\|=\|I+L\|\le 1+\|L\|
  \end{align*}
  we obtain a control of the action of $\hat B_n$ on $\hat
  F^B$ in a similar way to what we had for $B_n$ on $F^B$,
  were the constant $C$ is replaced by $C(1+\|L\|)$ on the
  upper bound and $C(1-\|L\|)$ on the lower bound, as long
  as we take $\dist(V^A,V^B)$ close enough to zero by
  letting $\delta_0$ be small enough since the beginning. An
  analogous estimate holds for $\hat B_n\mid \hat G_B$.
  
  We now estimate the angle between $\hat F^B$ and $\hat
  G^B$, that is, the new value which plays the role of $d$
  in Lemma~\ref{le:seqmatrixesHolder}. 

  We assume that for all $v\in F^B, w\in G^B$ we have
  $\|v\|\le d\|v+w\|$. If we now take $v\in\hat F^B,
  w\in\hat G^B$, then $\phi(v)\in F^B$ and $\phi(w)\in G^B$
  and so $\|\phi(v)\|\le d\|\phi(v+w)\|$. This implies
  $\|\phi^{-1}\|^{-1}\|v\|\le d \|\phi\|\|v+w\|$ so we
  obtain
  \begin{align*}
    \|v\|
    \le 
    d\|\phi\|\|\phi^{-1}\| \|v+w\| 
    \le
    \tau(d, L)\|v+w\|
    \quad\text{where}\quad
    \tau(d,L):=d \frac{1+\|L\|}{1-\|L\|}.
  \end{align*}
  The relation~\eqref{eq:hatF1AB} between $\|L\|$ and
  $\dist(V^A,V^B)$ enables us to deduce
  \begin{align*}
    \frac{1+\dist(V^A,V^B)}{1-\dist(V^A,V^B)} 
    \le
    \frac{1+\|L\|}{1-\|L\|} 
    \le
    \frac{1+\dist(V^A,V^B)\big(1-\dist(V^A,V^B)^2\big)^{-1/2}}
    {1-\dist(V^A,V^B)\big(1-\dist(V^A,V^B)^2\big)^{-1/2}}
  \end{align*}
  and so $\tau(d,L)$ is a function of
  $\dist(V^A,V^B)$. Moreover $\tau(d,L)$ goes to $d$ as
  $\dist(V^A,V^B)$ goes to zero. In particular, if
  $0<\|L\|<1/2$ we get $\tau(d,L)\le 3d$.

  We finally check the condition on the distance between
  $A_n$ and $\hat B_n$
  \begin{align*}
    \|A_n-\hat B_n\|
    &\le
    \|A_n-\phi^{-1}A_n\|+\|\phi^{-1}A_n - \phi^{-1}A_n\phi\|
    +
    \|\phi^{-1}A_n\phi - \hat B_n\|
    \\
    &\le
    \|I-\phi\| \| A_n\| + \|\phi^{-1}\|\|A_n\|\|I-\phi\|
    +\|\phi^{-1}\|\|A_n-B_n\|\|\phi\|
    \\
    &\le
    \frac{a^{|n|}\|L\|}{1-\|L\|}
    +
    \frac{a^{|n|}\|L\|}{(1-\|L\|)^2}
    +
    \delta a^{|n|} \frac{\|L\|(1+\|L\|)}{1-\|L\|}
    =a^n\|L\|\omega(L)
  \end{align*}
  since $ \|I-\phi^{-1}\|=\|I-(I+L)^{-1}\| \le
  \frac{\|L\|}{1-\|L\|}$, where $\|L\|$ is essentially
  $\dist(V^A,V^B)$ according to~(\ref{eq:hatF1AB}), and
  $\omega(L)$ can be made as close to $1$ as we need by
  letting $\|L\|$ (or $\dist(V^A,V^B)$) be close enough
  to zero.  Hence by taking $\delta$ small enough since the
  beginning we can ensure that $\|L\|\in(0,1/2)$ and
  $\|L\|\omega(L)\in(0,1/2)$.

  We can now apply Lemma~\ref{le:seqmatrixesHolder} to the
  pairs $F^A,G^A$ and $\hat F^B, \hat G^B$ to conclude
  \begin{align*}
    \dist(F^A,\hat F^B)\le (2+\tau(d,L))(C(1+\|L\|))^2
    \frac{\sigma_1}{\mu_2}
    (\|L\|\omega(L))^{\log(\sigma_1/\mu_2)/\log(a/\mu_2)}.
  \end{align*}
  For $0<\delta<\delta_0<1$ with $\delta_0$ small enough,
  depending on $C,\mu_1/\lambda_2$ and $a$, we obtain that
  $\|L\|\omega(L)\le 2\dist(V^A,V^B)$. Thus, setting
  $\eta=\log(\sigma_1/\mu_2)/\log(a/\mu_2)$ for
  simplicity
  \begin{align}\label{eq:betanu}
    \dist(F^A,\hat F^B)\le
    (2+3d)(\frac32 C)^2\cdot 2^\eta\frac{\sigma_1}{\mu_2}
    \dist(V^A,V^B)^{\eta}
    \le
    \frac94 C^2 (2+3d) \tilde C^\eta
    \frac{\sigma_1\mu_1^{\eta}}{\mu_2\lambda_2^{\eta}}
    \delta^{\omega\eta}.
  \end{align}

  To finish we estimate $\dist(\hat F^B,F^B)$.  

  We already have that $\dist(\phi^{-1}F^B,F^B)$ is
  comparable to $\|\phi^{-1}\|$ by \eqref{eq:normdist}.  We
  remark that
  $\phi^{-1}=(I+L)^{-1}=I+\sum_{k>0}(-L)^k=I+\hat L$, where
  $\hat L: V^A\to (V^A)^\perp$ and $\|\hat L\|\le
  \|L\|(1-\|L\|)^{-1}\le \|L\|\omega(L)$.  Hence
  \begin{align}\label{eq:hatdistance}
    \dist(\hat F^B,F^B) \le \|\hat L\| \le \|L\|\omega(L)
    \le 2\dist(V^A,V^B)
  \end{align}
  and we finally complete~\eqref{eq:triangle} using
  \eqref{eq:betanu} and \eqref{eq:beta}
  \begin{align}\label{eq:betagama}
    \dist(F^A,F^B) 
    \le 
    \frac94 C^2 (2+3d) \tilde C^\eta
    \frac{\sigma_1\mu_1^{\eta}}{\mu_2\lambda_2^{\eta}}
    \delta^{\omega\eta}
    +
    \widetilde C\frac{\mu_1}{\lambda_2} 
    \delta^\omega
    \le 
    \frac92 (2+3d)^{1+\eta}C^{2(1+\eta)}
    \frac{\sigma_1\mu_1^{\eta}}{\mu_2\lambda_2^{\eta}}
    \delta^{\omega\eta}
  \end{align}
  for $\delta_0\in (0,1)$ small enough. Putting the
  inequalities \eqref{eq:alfa}, \eqref{eq:gama} and
  \eqref{eq:betagama} together, we arrive to the conclusion
  of the lemma.
\end{proof}

\section{The Kontsevich-Zorich cocycle}
\label{sec:KZ}

In this section, we prove the theorem about H\"older
continuity of the Oseledets subspaces for the
Kontsevich-Zorich cocycle.

We first provide some general lemmas to deal with
topologically non-trivial cocycles.  The first is a
different point of view on Lemma \ref{le:seqmatrixesHolder};
the focus is on norms rather than matrices.  The second
allows one to intersect Oseledets filtrations, thus reducing
the necessary estimates.  Finally, the third allows one to
reduce vector bundle cocycles to matrix ones.

In the last part of the section, we combine these
constructions with the norm constructed by Avila, Gouezel
and Yoccoz \cite{AGY,AG}.  This gives the main results for
the Kontsevich-Zorich cocycle.

\subsection{Some more preliminary lemmas}

We consider Euclidean inner products on $\bR^d$, which we call metrics below.
Given two such, say $h_1$ and $h_2$, define the distance between them as
\begin{align}
\label{eqn:dist_metric}
\dist(h_1,h_2):=\log \sup_{\substack{\norm{v_1}_{h_1}=1 \\ \norm{v_2}_{h_2}=1 }} \left\{ \norm{v_1}_{h_2}, \norm{v_2}_{h_1} \right\}
\end{align}
This gives the inequality
$$
\norm{v}_{h_1}\leq e^{\dist(h_1,h_2)} \norm{v}_{h_2}
$$

Given a metric $h$, define the distance between two spaces $E,F$ to be
\begin{align*}
\dist_h(E,F):=\sup & \left\{ \norm{f^\perp}_h : e=f + f^\perp, \norm{e}_h=1, e\in E, f\in F, f^\perp\in F^\perp\right\} \bigcup \\
& \bigcup \left\{ \norm{e^\perp}_h : f=e + e^\perp, \norm{f}_h=1, f\in F, e\in E, e^\perp\in E^\perp\right\}.
\end{align*}

\begin{lemma}
\label{lemma:metrics_dist_bound}
 Suppose we are given decompositions of $\bR^d$ as
 $$
 \bR^d = E\oplus E^\perp = F\oplus F^\perp
 $$
 where perpendiculars are for a fixed metric $h_0$.  Suppose
 given a sequence of metrics $h_n^E, h_n^F$ satisfying
 \begin{align}
   \label{eqn:metric_assumption}
 \dist(h_n^E, h_n^F)\leq \log (1+\delta C_2 A^n)
\end{align}
 for some constants $C_2, A, \delta$.  Assume further that
 we have constants $0<\lambda<\mu$ and $C$ such that
 $\forall f\in F, \forall f'\in F^\perp$
 \begin{align}
   \norm{f}_{h_n^F}&\leq C \lambda^n \norm{f}_{h_0} \label{eqn:F_estimate}\\
  C^{-1} \mu^n \norm{f'}_{h_0} \leq \norm{f'}_{h_n^F} \label{eqn:F_estimate_down}.
 \end{align}
 Suppose that the above growth estimates hold analogously for $E$ instead of $F$.
 Then we have the distance estimate
 $$
 \dist_{h_0}(E,F)\leq C^2(2+C_2A) \delta^{\frac {\log \frac \mu \lambda} {\log A}}.
 $$
\end{lemma}

\begin{remark}
\label{rmk:fixed_time_bound}
In fact, in the proof below we do not need the inequalities
in the assumptions to hold for all $n$.  We only need to
know it for a value of $n$ such that $c_3\leq \delta A^n\leq
c_4$, for some fixed constants $c_3,c_4$.  Then the same
conclusion holds (with different constants, but same
$\delta$ and exponents).
\end{remark}

\begin{proof}
 Pick $e\in E$ with $\norm{e}_{h_0}=1$ and write it as $e=f+f^\perp$ where $f\in F, f^\perp\in F^\perp$.
 By the assumption on $F$ we have
 $$
 C^{-1}\mu^n \norm{f^\perp}_{h_0}\leq \norm{f^\perp}_{h_n^F}.
 $$
 We also have the following chain of inequalities
 \begin{align*}
  \norm{f^\perp}_{h_n^F}&= \norm{e-f}_{h_n^F} \leq \norm{e}_{h_n^F}+\norm{f}_{h_n^F}\\
  &\leq \norm{e}_{h_n^E}(1+\delta C_2 A^n) + C \lambda^n \norm{f}_{h_0} 
  && \textrm{by }\eqref{eqn:metric_assumption}\textrm{ for }\norm{e} \textrm{ and }\eqref{eqn:F_estimate} \textrm{ for }\norm{f}\\
  &\leq \norm{e}_{h_0} C\lambda^n (1+\delta C_2 A^n) + C \lambda^n
  && \textrm{by }\eqref{eqn:F_estimate}\textrm{ for }\norm{e} \textrm{ and since }\norm{f}_{h_0}\leq \norm{e}_{h_0}=1.
 \end{align*}
 Combining the two inequalities, we find that
 \begin{align*}
  \norm{f^\perp}_{h_0} &\leq 2C^2\left(\frac \lambda \mu \right)^n + C^2 C_2 \delta \left(\frac {\lambda A}\mu\right)^n\\
  & \leq C^2 \left(\frac \lambda \mu \right)^n (2+\delta C_2 A^n ).
 \end{align*}
 Choose now $n$ such that $1\leq \delta A^n<A$, i.e. $0\leq \log \delta + n\log A < \log A$.
 Then we have
 \begin{align*}
  \norm{f^\perp}_{h_0} & \leq C^2 (2+C_2A) e^{\left(\log \frac \lambda \mu \right) \cdot \frac {-\log \delta}{\log A}}\\
  &\leq C^2(2+C_2A) \delta^{\frac {\log \frac \mu \lambda} {\log A}}
 \end{align*}
 An analogous argument gives the estimate for vectors in $F$, thus proving the claim.
\end{proof}

\begin{remark}
 Suppose given two sequences of matrices $A_n, B_n$ with
 $$
 \norm{A_n-B_n} \leq \delta L^n
 $$
 Then we have the bound $ \norm{A_n v}\leq \norm{B_n v} + \delta L^n\norm{v}$.
 If we also had the Oseledets-type estimate
 $$
 \norm{v}\leq C\mu^{-n}\norm{B_n v}
 $$
 we could conclude $ \norm{A_n v} \leq \norm{B_n v}(1+\delta C L^n \mu^{-n} )$.
 If we defined a sequence of metrics by
 \begin{align*}
  \norm{v}_{h_n^A}&:= \norm{A_n v}\\
  \norm{v}_{h_n^B}&:= \norm{B_n v}
 \end{align*}
 then we could apply the above lemma (with $A=\frac L \mu$) and get a conclusion similar to that of Lemma 5.3.3 from \cite{BarPes2007}. 
 The advantage of using the metric formulation is that the metrics defined above can come from linear operators with the same source, but different targets.
\end{remark}

The next result allows us to intersect subspaces in the past and future Oseledets filtrations.

\begin{lemma}
\label{lemma:Grass_intersection}
 Let $K$ be a compact metric space and let $F_+,F_-$ be two H\"older continuous maps to Grassmanians
 $$
 F_\pm:K \to Gr(k_{\pm},\bR^d)
 $$
 with H\"older exponents $\nu_+,\nu_-$.  Suppose that for
 all $x\in K$, we have that $F_+(x)$ and $F_-(x)$ are in
 general position, i.e. intersect in a space of dimension
 $r:=k_+ + k_- - d$.
 
 Then the map
 $$
 I:K\to Gr(r,\bR^d)
 $$
 assigning to $x$ the intersection of $F_+(x)$ and $F_-(x)$
 is H\"older of exponent $\min(\nu_+,\nu_-)$.
\end{lemma}

\begin{proof}
 We omit the $\bR^d$ from the notation for Grassmanians.
 Inside $Gr(k_+)\times Gr(k_-)\times Gr(r)$ define the closed algebraic set
 $$
 \Gamma:=\left\{ (V_+,V_-,I) : I\subseteq V_+\cap V_-\right\}.
 $$
 We have projections
 \begin{align*}
  p_{12}&:\Gamma \to Gr(k_+)\times Gr(k_-)\\
  p_3 & :\Gamma \to Gr(r).
 \end{align*}
 Over the Zariski-open set $U\subset Gr(k_+)\times Gr(k_-)$
 of planes in general position we know that the map
 $$
 p_{12}: \left(\Gamma \cap p^{-1}_{12}(U)\right) \to U
 $$
 is a smooth (even algebraic) bijection.  Denote its inverse
 by $p_{12,U}^{-1}$ assumed to take values in $Gr(k_+)\times
 Gr(k_-)\times Gr(r)$.
 
 By assumption, the map $F_+\times F_- : K\to U$ lands
 inside the open set, and has compact image.  In particular,
 the first derivative of $p_{12,U}^{-1}$ is uniformly
 bounded on this image.
 
 The function $I$ giving the intersection of $F_+$ and $F_-$
 over $K$ can now be written as the composition $p_3\circ
 p_{12,U}^{-1}\circ (F_+\times F_-)$.  It is therefore
 H\"older, with exponent $\min (\nu_+,\nu_-)$.
\end{proof}

We next explain how cocycles on vector bundles can be
reduced to cocycles over matrices.

\begin{lemma}
\label{lemma:isometric_trivialization}
Suppose $V\to X$ is a vector bundle with metric over a
connected manifold $X$.  Then we can find another vector
bundle with metric $W\to X$ and an isomorphism $V\oplus
W\cong \underline{\bR}^N$.
 
Here $\underline{\bR}^N$ denotes the trivial vector bundle
over $X$, with the standard metric.  The regularity of $W$
and the isomorphism can be taken as good as that of $V$ and
$X$.
\end{lemma}

\begin{proof}
  First, we show that $V$ can be trivialized using finitely
  many open sets (not necessarily connected).  Then we find
  the isometric embedding.
 
  To begin, since $X$ is a manifold it has finite Lebesgue
  covering dimension, say $n$.  Therefore, there exist a
  cover of $X$ by open sets $U_\alpha$ such that at most
  $n+1$ intersect at any point, and we have trivializations
  $U_\alpha\times \bR^d\cong \left. V \right|_{U_\alpha}$.
 
  Take a partition of unity $\phi_\alpha$ subordinated to
  the covering, and for any set
  $I=\{\alpha_1,\ldots,\alpha_i\}$ define the open sets
 $$
 W_I:=\{x\in X : \phi_{\alpha}(x) < \phi_{\alpha_j}(x), \forall \alpha\notin I, \alpha_j\in I  \}
 $$
 The sets $W_I$ are open by the local finiteness of the
 covering, and $V$ can be trivialized on each of them.
 Furthermore, if $I\neq I'$ and of the same cardinality,
 then $W_I\cap W_{I'}=\emptyset$.  We can thus define the
 $n+1$ open sets
 $$
 W_i := \bigcup_{\# I =i}W_I \hskip 1cm i=1,\ldots,n+1
 $$
 These sets are open, give a cover of $X$ and $V$ is
 trivialized on each of them.
 
 Next, take a partition of unity $\rho_i$ adapted to $W_i$.
 Let $v_1^{(i)},\ldots, v_d^{(i)}$ be the global sections of
 $V$ obtained from the coordinate sections in each
 trivialization, multiplied by the $\rho_i$ to have support
 in $W_i$.
 
 We can now define a surjection
 \begin{align*}
   X\times \bR^{d(n+1)} \onto &V\\
   \overrightarrow{e}_{(i-1)d+j} \mapsto &v_j^{(i)} \hskip
   0.5in i=1,\ldots,n+1, \hskip 0.5in j=1,\ldots,d
 \end{align*}
 Here $\overrightarrow{e}_l$ denotes the $l$-th coordinate
 vector in $\bR^{d(n+1)}$.
 
 Let $W$ be the kernel bundle of the surjection just
 defined, and let $V'$ be its orthogonal complement.  The
 surjection induces an isomorphism $V'\to V$, which we can
 invert to an injection $V\to \underline{\bR^N}$.
 
 Letting $h$ be the initial metric on $V$, the embedding
 just described induces a different metric $h'$.  The space
 of metrics on a fiber of the bundle is
 $\SL_d(\bR)/SO_d(\bR)$, which is contractible.  Moreover,
 we have unique geodesics given as exponentials of symmetric
 operators (symmetric for the underlying metric).
 See for example \cite[pg. 324]{BH_npc} for a discussion of these properties.
 
 This means that precomposing the injection $V\to
 \underline{\bR}^N$ with uniquely defined linear operators
 on the fibers of $V$, we can arrange the map to be
 isometric.  We have thus obtained the desired isometric
 decomposition $ V\oplus W = \underline{\bR}^N $
\end{proof}

\subsection{Applications to the Kontsevich-Zorich cocycle}

We refer to \cite{Zorich_survey} for a general introduction
to flat surfaces and the Kontsevich-Zorich (KZ) cocycle.
The results of Eskin, Mirzakhani, and Mohammadi
\cite{EM,EMM} lead us to consider affine invariant
manifolds, since those support the $\SL_2\bR$-invariant
measures.

In this section, we prove two types of results in that
context.  The first one is that the main theorem
\ref{mthm:Holder-bundles-vectorbundle} applies to the KZ
cocycle.  The second one is similar to the results of
Chaika-Eskin \cite{Chaika-Eskin}.

\begin{theorem}
\label{thm:KZ_global}
Let $\cM$ be an affine invariant manifold and let $\mu$ be
the corresponding ergodic $\SL_2\bR$-invariant probability
measure (see \cite{EM}).  Let $E$ be the Kontsevich-Zorich
cocycle (or any of its tensor powers) and let
$\{\lambda_i\}$ be its Lyapunov exponents, with Oseledets
subspaces $E^{i}$.
 
Then there exists $\nu_i>0$ such that for any $\epsilon>0$
there exists a compact set $K_\epsilon$ with
$\mu(K_\epsilon)>1-\epsilon$ and such that the spaces $E^i$
vary $\nu_i$-H\"older continuously on $K_\epsilon$.
\end{theorem}

\paragraph{\textbf{The AGY norm}}
Before the start the proof, we recall some properties of a norm defined by Avila-Gou\"ezel-Yoccoz (see \cite[Sect. 2.2.2]{AGY} and \cite[Sect. 5]{AG}).
For the definitions below, we potentially need to pass to a finite cover of $\cM$ to avoid orbifold issues.

Denote by $H^1_{rel}$ the real bundle of relative cohomology.
We denote by $\omega$ both a point in the affine manifold $\cM$ and the cohomology class in $H^1_{rel}$ that it represents.
The norm on $H^1_{rel}$ is then defined by
\begin{align}
  \label{eqn:AGY_def}
\norm{\alpha}_\omega := \sup_{\textrm{saddle } \gamma} \frac{|\alpha(\gamma)|}{|\omega(\gamma)|}
\end{align}
Recall that $\gamma$ is \emph{a saddle} if on the flat surface represented by $\omega$, the class $\gamma$ can be realized as a straight line connecting two singular points of the flat metric.

Identifying the relative cohomology with the tangent space to the stratum, the expression \eqref{eqn:AGY_def} gives a complete metric (\cite[Cor. 2.13]{AGY}).
Given $x\in \cM$, denote by $W^u(x)$ the unstable leaf through $x$, and by $E^u(x)\subseteq H^1_{rel}(x)$ the tangent space to the unstable leaf.
We have the exponential map (linear in period coordinates)
\begin{align*}
 \Psi_x:H^1_{rel}(x) &\to W^u(x)\\
 \alpha & \mapsto x+\alpha
\end{align*}

The main properties of the metric relative to the Teichm\"uller flow are summarized below.
We denote by $Dg_t$ the induced cocycle on the tangent bundle, and by $g_t$ the Kontsevich-Zorich cocycle (i.e. the flat Gauss-Manin connection).

\begin{enumerate}
 \item \cite[ineq. 2.13]{AGY} For $x\in \cM$, $v\in H^1_{rel}(x)$ we have the growth bounds
 \begin{align}
   \label{eqn:AGY_top_exp_bound}
 e^{-2t}\norm{v}_x \leq \norm{Dg_t v}_{g_tx}\leq e^{2t}\norm{v}_x
\end{align}
 \item \cite[Prop. 5.3]{AG} If we denote by $B(0,r)_x$ the ball of radius $r$ in $H^1_{rel}(x)$ then $\Psi_x$ is well-defined on $B(0,1/2)$ for all $x\in \cM$, and for all $v\in B(0,1/2)$ we have
 \begin{align}
 d_{W^u(x)}(x,\Psi_x(v)) \leq 2\norm{v}_x
\end{align}
 \item \cite[Prop. 5.3]{AG} For all $w\in E^u(x)$ and $v\in B(0,1/2)$ we have
 \begin{align}
   \label{eqn:AGY_local_bound}
 \frac 12 \norm{w}_{\Psi_x(v)} \leq \norm{w}_x \leq 2 \norm{w}_{\Psi_x(v)}
\end{align}
 In the above inequality, the vector $w$ is transported to $\Psi_x(v)$ using the flat connection along the exponential map.
\end{enumerate}

In order to invoke Lemma \ref{lemma:metrics_dist_bound}, we
note that its proof works in the same way for Finsler
metrics.  The distance between Finsler metrics is defined by
equation \eqref{eqn:dist_metric}.

Alternatively, to a Finsler metric one can associate canonically the Riemannian metric coming from the John ellipsoid construction (see \cite[Thm. 3.1]{ball_convex}), and each controls the other by absolute constants (depending on dimension only).
Recall that given a bounded convex body $K\subset \bR^n$, there is a canonically associated ellipsoid $E\subseteq K$ (the John ellipsoid) with the following two properties.
First, $E$ is the ellipsoid of maximal volume and satisfying $E\subseteq K$, and moreover $K\subseteq \sqrt{n} E$, where $\sqrt{n}E$ denotes the rescaling by $\sqrt{n}$ based at the center of $E$.
Interpreting a convex body as the unit ball of a norm, this means that we can replace a Finsler metric with a canonical inner product, such that each is within a bounded factor of the other.

\begin{proof}[Proof of Theorem \ref{thm:KZ_global}]
 We prove H\"older continuity for the forward Oseledets filtration, the proof for the backward one being similar.
 Combining them using Lemma \ref{lemma:Grass_intersection}, the desired claim follows.
 
 Define the forward Oseledets filtration by
 $$
 F^i(\theta)_x := \oplus_{i\geq j} E^j(\theta)_x
 $$ 
 Fix $\epsilon_1>0$ arbitrarily small.
 Then for $C$ sufficiently large, there exists a compact set $K_\epsilon\subseteq \cM$ of measure at least $1-\epsilon/10$ on which the Oseledets theorem holds uniformly:
 $$
 C^{-1} e^{(\lambda_i-\epsilon_1)t}\norm{v_i}\leq \norm{g_t v_i}\leq C e^{(\lambda_i+\epsilon_1)t} \norm{v_i}
 \hskip 0.5cm \forall i,\forall v_i \in E^i_x, \forall x\in K_\epsilon
 $$
 By passing to a further compact subset, we can assume the angle between Oseledets subspaces is uniformly bounded away from zero.
 
 Since the forward Oseledets filtration only depends on the position of the point on unstable leaves, given $x$ and $y\in W^u(x)$, it suffices to estimate $\dist(F^i(x),F^i(y))$ in terms of $d_{W^u(x)}(x,y)$.
 Assuming $d_{W^u(x)}(x,y)\leq 1/4$, using the properties of the AGY norm we have $y=\Psi_x(v)$ for some $v\in H^1_{rel}(x)$.
 
 Denote by $GM_v$ the Gauss-Manin connection from $x$ to $x+v$ along the exponential map.
 If $w\in E_x$ is a vector, we define new norms by flowing forward:
 $$
 \norm{w}_{x,t} := \norm{g_t w}_{g_tx}
 $$
 A similar norm is defined on $E_y$ and we wish to compare $\norm{w}_{x,t}$ and $\norm{GM_v(w)}_{y,t}$.
 
 Set $\delta=\norm{v}_x$ and $t=\frac {-1}2 \log\delta - 2$.
 Note that with this choice we have $1/100<\delta e^{2t}< 1/4$.
 We would like to apply Lemma \ref{lemma:metrics_dist_bound} (see Remark \ref{rmk:fixed_time_bound} for why it suffices to check assumptions at a well-chosen time).
 The Oseledets-type assumptions of Lemma \ref{lemma:metrics_dist_bound} (i.e. \eqref{eqn:F_estimate} and \eqref{eqn:F_estimate_down}) on the filtrations hold by the choice of compact sets.
 We need to estimate the distance between norms, i.e. check assumption \eqref{eqn:metric_assumption} in Lemma \ref{lemma:metrics_dist_bound}.
 
 Thus, given $w\in E_x$, we need to estimate the ratio of $\norm{g_t w}_{g_t x}$ and $\norm{GM_{Dg_t v}(g_t w)}_{g_t y}$.
 By our choice of $t$ and using property \eqref{eqn:AGY_top_exp_bound} on the AGY norm, we have $\norm{Dg_t v}_{g_t x}\leq 1/2$
 Therefore, using property \eqref{eqn:AGY_local_bound} of the AGY norm we find
 $$
 \frac 12 \norm{GM_{Dg_t v} (g_t w)}_{g_ty} \leq  \norm{g_t w}_{g_t x}  \leq 2 \norm{GM_{Dg_t v} (g_t w)}_{g_ty}
 $$
 Therefore, the assumptions of Lemma \ref{lemma:metrics_dist_bound} are satisfied, with $A=e^2$ (coming from property \eqref{eqn:AGY_top_exp_bound} of the AGY norm).
 We can now conclude the forward Oseledets filtration $F^i$ varies H\"older-continuously with exponent $\frac 1{2} \log \frac {\lambda_i}{\lambda_{i+1}}- \epsilon'$.
 Combining the forward and backward filtrations using Lemma \ref{lemma:Grass_intersection}, we find individual Oseledets subspaces $E^i$ are H\"older with exponent $\nu_i:=\frac 1{2} \min \left(\log \frac {\lambda_i}{\lambda_{i+1}}, \log \frac {\lambda_{i-1}}{\lambda_i} \right) - \epsilon'$.
\end{proof}

Following Chaika-Eskin \cite{Chaika-Eskin}, the next result refines the above one when restricting to a single $\SL_2\bR$-orbit.
Recall the notation for $1$-parameter subgroups of $\SL_2\bR$:
\begin{align*}
 g_t:=\begin{bmatrix}
       e^t & 0 \\
       0 & e^{-t}
      \end{bmatrix}
      \hskip 0.5cm
 r_\theta := \begin{bmatrix}
              \cos \theta & -\sin \theta\\
              \sin \theta & \cos \theta
             \end{bmatrix}
      \hskip 0.5cm
 g_t^\theta:= r_{\theta}^{-1} g_t r_\theta & 
\end{align*}

\begin{theorem}
\label{thm:KZ_disk}
 Let $x\in \cH(\kappa)$ be a flat surface in a stratum and let $E$ be the Kontsevich-Zorich cocycle.
 If $E$ is some tensor power of the KZ cocycle, let $k\geq 1$ be the smallest order of the tensor product which contains it.
 
 For $g\in \SL_2\bR$ denote by $A(g,x):E_x \to E_{gx}$ the matrix of the cocycle $E$.
 By Theorem 1.2 in \cite{Chaika-Eskin} there exist $\lambda_1>\cdots >\lambda_k$ such that for a.e. $\theta\in S^1$ we have \begin{enumerate}
  \item There exists a decomposition $E_x = \oplus_i E^{i}(\theta)_x$ with $E^{i}(\theta)_x$ measurably varying in $\theta$
  \item We have
  \begin{align*}
   \lim_{t\to \pm \infty} \frac 1 t \log \norm{A(g_t^\theta, x) v_i} = \lambda_i \hskip 0.5cm \forall v_i\in E^i(\theta)_x
  \end{align*}
 \end{enumerate}
 Let $\nu_i:=\frac 1{2k} \min \left(\log \frac {\lambda_i}{\lambda_{i+1}}, \log \frac {\lambda_{i-1}}{\lambda_i} \right) - \epsilon'$, for arbitrarily small $\epsilon'$.
 
 Then for any $\epsilon>0$ there exists a compact set $K_\epsilon\subseteq S^1$ of Lebesgue measure at least $1-\epsilon$ such that the subspaces $E_i(\theta)$ vary $\nu_i$-H\"older continuously for $\theta\in K_\epsilon$. 
\end{theorem}

\begin{proof}
 First we show that the forward (resp. backward) Oseledets filtrations are H\"older.
 Then we combine the results to conclude individual subspaces are H\"older.
 
 Define the forward Oseledets filtration by
 $$
 F^i(\theta)_x := \oplus_{i\geq j} E^j(\theta)_x
 $$
 Define also norms on $E_x$ by
 $$
 \norm{v}_{g_t^\theta}:= \norm{A(g_t^\theta,x)v} 
 $$
 
 Fix $\epsilon_1>0$ arbitrarily small.
 Then for $C$ sufficiently large, there exists a compact set $K_1\subseteq S^1$ of measure at least $1-\epsilon/10$ on which the Oseledets theorem holds uniformly:
 $$
 C^{-1} e^{(\lambda_i-\epsilon_1)t}\norm{v_i}\leq \norm{v_i}_{g_t^\theta}\leq C e^{(\lambda_i+\epsilon_1)t} \norm{v_i}
 \hskip 0.5cm \forall i,\forall v_i \in E^i(\theta)_x, \forall \theta\in K_1
 $$
 By Lusin's theorem, we can further restrict to a compact subset $K_2\subseteq K_1$ of measure at least $1-\epsilon$ such that $E_i(\theta)$ vary continuously on it.
 In particular, the distance between $E^i(\theta)$ and $E^j(\theta)$ for $i\neq j$ is uniformly bounded away from zero on $K_2$.
 
 Consider now the decomposition $E=F^i(\theta)\oplus F^i(\theta)^\perp$.
 From the uniform boundedness of the angle between Oseledets subspaces, there exists $C_1>0$ such that
 \begin{align*}
    \norm{v}_{g_t^\theta} \leq C_1 e^{(\lambda_i+\epsilon_1)t} \norm{v}  \hskip 1in \forall v\in F^i(\theta)\\
 C_1^{-1} e^{(\lambda_{i-1}-\epsilon_1)t} \leq \norm{v}_{g_t^\theta}     \hskip 1in \forall v \in F^i(\theta)^\perp 
 \end{align*}
 We can now apply Lemma \ref{lemma:metrics_dist_bound} to conclude the H\"older continuity of $F^i(\theta)$, provided we control the divergence of metrics.
 We need to find a $A,C_2$ such that
 $$
 \dist\left(\norm{-}_{g_t^{\theta_1}}, \norm{-}_{g_t^{\theta_2}}\right) \leq A |\theta_1-\theta_2| e^{tC_2}
 $$
 
 By results of Forni (see \cite{Forni_deviations}) the norm of the Kontsevich-Zorich cocycle on a Teichm\"uller disk is bounded by the hyperbolic distance between points.
 Namely, if we $E$ is contained in a $k$-th tensor product of the KZ cocycle, then
 $$
 \norm{A(x,gx)}\leq e^{k\cdot \dist(x,gx)}
 $$
 
 Next, from hyperbolic geometry we have for $t$ bounded away from zero
 $$
 e^{\frac 12 \dist\left(g_t^{\theta_1}x, g_t^{\theta_2} x\right)} \leq C_3 |\theta_1-\theta_2| e^t
 $$
 Plugging in these estimates into Lemma \ref{lemma:metrics_dist_bound}, we get H\"older continuity of $F^i(\theta)$ with exponent $\nu_i = \frac 1{2k} \log \frac{\lambda_{i-1}} {\lambda_i}$.
 
 An analogous result holds for the backward Oseledets filtration.
 Therefore, combining the filtrations using Lemma \ref{lemma:Grass_intersection} we obtain the desired result.
\end{proof}

%%%%%%%%%%%%%%%%%%%%%%%%%%%%%%%%%%%%%%%%%%%%%%%%%%%%%%%%%%%%%%%

\def\cprime{$'$}

\bibliographystyle{alpha}

\begin{thebibliography}{EMM13}

\bibitem[AG13]{AG}
Artur Avila and S{\'e}bastien Gou{\"e}zel.
\newblock {Small eigenvalues of the {L}aplacian for algebraic measures in
  moduli space, and mixing properties of the {T}eichm{\"u}ller flow}.
\newblock {\em Ann. of Math. (2)}, 178(2):385--442, 2013.

\bibitem[AGY06]{AGY}
Artur Avila, S{\'e}bastien Gou{\"e}zel, and Jean-Christophe Yoccoz.
\newblock {Exponential mixing for the {T}eichm{\"u}ller flow}.
\newblock {\em Publ. Math. Inst. Hautes {\'E}tudes Sci.}, 104:143--211, 2006.

\bibitem[Bal97]{ball_convex}
Keith Ball.
\newblock {An elementary introduction to modern convex geometry}.
\newblock In {\em {Flavors of geometry}}, volume~31 of {\em {Math. Sci. Res.
  Inst. Publ.}}, pages 1--58. Cambridge Univ. Press, Cambridge, 1997.

\bibitem[BH99]{BH_npc}
Martin~R. Bridson and Andr{\'e} Haefliger.
\newblock {\em {Metric spaces of non-positive curvature}}, volume 319 of {\em
  {Grundlehren der Mathematischen Wissenschaften [Fundamental Principles of
  Mathematical Sciences]}}.
\newblock Springer-Verlag, Berlin, 1999.

\bibitem[BP07]{BarPes2007}
Luis Barreira and Yakov Pesin.
\newblock {\em {Nonuniform hyperbolicity}}, volume 115 of {\em {Encyclopedia of
  Mathematics and its Applications}}.
\newblock Cambridge University Press, Cambridge, 2007.
\newblock Dynamics of systems with nonzero Lyapunov exponents.

\bibitem[Bri01]{brin-app}
M.~Brin.
\newblock {Proc. Sympos. Pure Math.}
\newblock In {\em {Smooth Ergodic Theory and its Applications}}, volume~{69},
  chapter {Appendix: H{\"o}lder continuity of invariant distributions}, pages
  {91--95}. {Amer. Math. Soc.}, {2001}.

\bibitem[CE15]{Chaika-Eskin}
Jon Chaika and Alex Eskin.
\newblock {Every flat surface is Birkhoff and Oseledets generic in almost every
  direction}.
\newblock {\em Journal of Modern Dynamics}, 9(01):1--23, 2015.

\bibitem[EM13]{EM}
A.~{Eskin} and M.~{Mirzakhani}.
\newblock {Invariant and stationary measures for the SL(2,R) action on Moduli
  space}.
\newblock {\em ArXiv e-prints} \texttt{http://arxiv.org/abs/1302.3320}, 2014.

\bibitem[EMM13]{EMM}
A.~{Eskin}, M.~{Mirzakhani}, and A.~{Mohammadi}.
\newblock {Isolation, equidistribution, and orbit closures for the SL(2,R)
  action on Moduli space}.
\newblock {\em Ann. of Math. (2)}, 182(2): 673-721, 2015.

\bibitem[For02]{Forni_deviations}
Giovanni Forni.
\newblock {Deviation of ergodic averages for area-preserving flows on surfaces
  of higher genus}.
\newblock {\em Ann. of Math. (2)}, 155(1):1--103, 2002.

\bibitem[Ma{\~n}87]{Man87}
Ricardo Ma{\~n}{\'e}.
\newblock {\em {Ergodic theory and differentiable dynamics}}, volume~8 of {\em
  {Ergebnisse der Mathematik und ihrer Grenzgebiete (3) [Results in Mathematics
  and Related Areas (3)]}}.
\newblock Springer-Verlag, Berlin, 1987.
\newblock Translated from the Portuguese by Silvio Levy.

\bibitem[Ose68]{Os68}
V.~I. Oseledets.
\newblock {A multiplicative ergodic theorem: Lyapunov characteristic numbers
  for dynamical systems}.
\newblock {\em {Trans. Moscow Math. Soc.}}, {19}:{197--231}, {1968}.

\bibitem[Pes77]{Pe77}
Ya.~B. Pesin.
\newblock {Characteristic Lyapunov exponents and smooth ergodic theory}.
\newblock {\em {Russian Math. Surveys}}, {324}:{55--114}, {1977}.

\bibitem[PS71]{PuSh71}
Charles Pugh and Michael Shub.
\newblock {Ergodic elements of ergodic actions}.
\newblock {\em Compositio Math.}, 23:115--122, 1971.

\bibitem[Zor06]{Zorich_survey}
Anton Zorich.
\newblock {Flat surfaces}.
\newblock In {\em {Frontiers in number theory, physics, and geometry. {I}}},
  pages 437--583. Springer, Berlin, 2006.

\end{thebibliography}

\end{document}